\documentclass[11pt,a4paper]{amsart}

\title{Non-associative Frobenius algebras of type $G_2$ and $F_4$}
\author{Jari Desmet}
\address{\parbox{\linewidth}{Ghent University \\ Department of Mathematics: Algebra and Geometry \\ Krijgslaan 281 -- S25 \\ 9000 Gent \\ Belgium}}
\email{\href{mailto:jari.desmet@ugent.be}{jari.desmet@ugent.be}}
\date{\today}
\makeatletter
\@namedef{subjclassname@2020}{%
  \textup{2020} Mathematics Subject Classification}
\makeatother

\usepackage[english]{babel}
\usepackage{amsmath,amssymb,amsthm}
\usepackage[utf8]{inputenc}
\usepackage{graphicx}
\usepackage{tikz}
\usetikzlibrary{calc}
\usepackage{tikzit}
\tikzstyle{node}=[fill=blue, draw=black, shape=circle]
\tikzstyle{new style 0}=[fill=none, draw=black, shape=circle]

\tikzstyle{new edge style 0}=[-, draw=blue, fill=none]
\tikzstyle{new edge style 1}=[->, draw=black]
\tikzstyle{new edge style 2}=[-, draw=red]
\tikzstyle{new edge style 3}=[-, draw=green]
\tikzstyle{new edge style 4}=[-, draw={rgb,255: red,255; green,128; blue,0}]
\tikzstyle{new edge style 5}=[-, draw=blue, fill={rgb,255: red,161; green,158; blue,255}]
\tikzstyle{new edge style 6}=[-, draw=red, fill={rgb,255: red,255; green,134; blue,146}]
\tikzstyle{new edge style 7}=[-, fill={rgb,255: red,153; green,255; blue,160}, draw=green]
\tikzstyle{new edge style 8}=[-, fill={rgb,255: red,246; green,255; blue,57}, draw=yellow]
\tikzstyle{new edge style 9}=[-, draw={rgb,255: red,128; green,0; blue,128}, fill={rgb,255: red,186; green,0; blue,186}]
\tikzstyle{new edge style 10}=[-, fill={rgb,255: red,255; green,183; blue,120}, draw={rgb,255: red,255; green,128; blue,0}]

\usepackage{arydshln}
\usepackage{tikzpagenodes}
\usepackage[shortlabels]{enumitem}
\setenumerate[0]{label={\rm (\roman*)}, leftmargin=*}
\usepackage{faktor}
\usepackage{mathtools}
\usepackage{breqn}
\usepackage{comment}

\usepackage[margin = 3cm]{geometry}
\usepackage[textsize=footnotesize,textwidth=15ex]{todonotes}

\usepackage{hyperref}

\usepackage[capitalise, noabbrev]{cleveref}

\theoremstyle{definition}
\newtheorem{definition}{Definition}[section]
\newtheorem{notation}[definition]{Notation}

\theoremstyle{plain}
\newtheorem*{theorem*}{Theorem}
\newtheorem{theorem}[definition]{Theorem}
\newtheorem{lemma}[definition]{Lemma}
\newtheorem{proposition}[definition]{Proposition}
\newtheorem{corollary}[definition]{Corollary}

\theoremstyle{remark}

\newtheorem{remark}[definition]{Remark}

\newcommand{\diff}{\textrm{d}}
\DeclareMathOperator{\Spin}{Spin}
\DeclareMathOperator{\Lie}{Lie}
\DeclareMathOperator{\im}{Im}

\DeclareMathOperator{\gl}{\mathfrak{\gl}}
\DeclareMathOperator{\Oct}{\mathbb{O}}

\DeclareMathOperator{\rk}{rk}
\DeclareMathOperator{\Tr}{Tr}
\newcommand{\g}{\mathfrak{g}}
\newcommand{\Ag}{A(\g)}
\DeclareMathOperator{\kar}{char}
\DeclareMathOperator{\ad}{ad}
\DeclareMathOperator{\Sq}{S}
\DeclareMathOperator{\End}{End}
\DeclareMathOperator{\Aut}{Aut}
\DeclareMathOperator{\GL}{GL}

\DeclareMathOperator{\sgn}{sgn}
\DeclareMathOperator{\SL}{SL}
\DeclareMathOperator{\SO}{SO}

\DeclareMathOperator{\Orth}{O}

\newcommand{\id}{\mathrm{id}}

\numberwithin{equation}{section}

\begin{document}
\maketitle

\begin{abstract}
    Very recently, Maurice Chayet and Skip Garibaldi have introduced a class of commutative non-associative algebras, for each simple linear algebraic group over an arbitrary field (with some minor restriction on the characteristic).
    We give an explicit description of these algebras for groups of type $G_2$ and $F_4$ in terms of the octonion algebras and the Albert algebras, respectively. As a byproduct, we determine all possible invariant commutative algebra products on the representation with highest weight $2\omega_1$ for $G_2$ and on the representation with highest weight $2\omega_4$ for $F_4$.
    
    It had already been observed by Chayet and Garibaldi that the automorphism group for the algebras for type $F_4$ is equal to the group of type $F_4$ itself. Using our new description, we are able to show that the same result holds for type $G_2$.
\end{abstract}

\section{Introduction}

Very recently, Maurice Chayet and Skip Garibaldi have introduced a class of commutative non-associative algebras, for each simple linear algebraic group over an arbitrary field (with some minor restriction on the characteristic); see \cite{chayet2020class}. This construction is quite remarkable, because it applies to all simple linear algebraic groups up to isogeny, regardless of type and form.

In particular, this provides a new $3875$-dimensional algebra, the automorphism group of which is precisely a group of type~$E_8$. This algebra is particularly interesting because it lives on the second smallest representation for $E_8$; the smallest one is the adjoint representation (of dimension $248$).

The same $3875$-dimensional algebra had also been constructed almost simultaneously by Tom De Medts and Michiel Van Couwenberghe in \cite{DMVC21}, but the construction of Chayet and Garibaldi in \cite{chayet2020class} is more general and includes the non-simply-laced case, in contrast to the results from \cite{DMVC21}. In addition, the results from both papers focus on different aspects of these algebras and are therefore largely complementary.

The construction of this class of algebras is explicit in the sense that it is constructed from the symmetric square of the Lie algebra. However, it seems desirable to find constructions of these algebras that do not already start from the adjoint representation but that use other algebraic structures instead. This is hinted at in \cite[Proposition~10.5]{chayet2020class}, but other than the tiny example of type $A_2$ in \cite[Example~10.9]{chayet2020class}, Chayet and Garibaldi do not investigate this further.

The aim of our paper is precisely to obtain an explicit description of these algebras for groups of type $G_2$ and $F_4$. Not surprisingly, we will use octonion algebras for the case $G_2$ and Albert algebras for the case $F_4$, but we need a substantial number of other ingredients to arrive at our final description. Our main results are \cref{identificationag2,algF4}.

Using our new description, we also prove that the automorphism group for the algebras of type $G_2$ is again of type $G_2$ (and nothing more), a fact that was already observed in \cite[Proposition~9.1]{chayet2020class} for the algebras for type $F_4$ and $E_8$. This is the content of \cref{autG2}.

\subsection*{Outline of this paper}
In \cref{preliminaries}, we first recall some basic representation theory and the results of \cite{chayet2020class}. Of these results, \cite[Proposition~10.5]{chayet2020class} will be the key to this article. After that we introduce the octonion algebras, the Albert algebras and all identities satisfied in these algebras that are needed in this paper.

Then, in \cref{subsecstandardder} we give some properties of \emph{standard derivations} of the octonion algebras and the Albert algebras, a concept introduced by Richard D. Schafer (see \cite[Identities~(3.70) and (4.6)]{nonassocalg}). These standard derivations will be the key to describing the algebras in a different way.

After a short section about symmetric operators we turn to the case of $G_2$, where we describe an isomorphism from the algebra to the symmetric square of the \emph{pure octonions}, the $7$-dimensional natural representation on which $G_2$ acts. This is done using the standard derivations from \cref{subsecstandardder}, and \cite[Proposition~10.5]{chayet2020class}, which provides us with an embedding $\sigma$ of the algebra into the endomorphism ring of the $7$-dimensional representation. The solution works in 4 steps:
	\begin{enumerate}
		\item\label{step1} Describe $\sigma$ explicitly using only the bilinear form and the Malcev product on the pure octonions, instead of the abstract Lie algebra,
		\item\label{step2} determine the underlying vector space image of $\sigma$,
		\item\label{step3} using the explicit image of $\sigma$ from \ref{step2} and character computations, determine all $G_2$-equivariant multiplications on this representation explicitly,
		\item\label{step4} determine an explicit formula for the multiplication defined by Chayet and Garibaldi, in terms of the multiplications found in \ref{step3}.
	\end{enumerate}

Using this novel description, we prove in \cref{secautg2} that the group of type $G_2$ is the full automorphism group of the algebra, using arguments similar to the ones in \cite{GG15}. Essentially this is a case-by-case elimination of all other possibilities.

In \cref{secf4}, we reuse the recipe proposed for $G_2$ to also describe the algebra for type $F_4$, but with the Albert algebras instead. Here, both step \ref{step1} and \ref{step3} come with more difficulties. In step \ref{step1}, it does not seem feasible to obtain a closed formula for the embedding, though we can easily compute the embedding for some particular elements in $A(\g)$. In step~\ref{step3}, the multiplications on the algebra are not as easily defined as for the $G_2$ case, and we need some new symmetric identities of degree $4$ that are satisfied by the Albert algebra to construct them.

\subsection*{Acknowledgements}
The author was supported by the BOF PhD mandate \linebreak BOF21/DOC/158 during part of this research, and is currently supported by the FWO PhD mandate 1172422N. The author is grateful to his supervisor Tom De Medts, for guiding him through the process of writing this article, and for providing simpler proofs for \cref{technicallem,standardderinprod}. The author also wants to thank Simon Rigby for providing substantial improvements to the exposition of this article, and Skip Garibaldi for clarifying comments about the positive characteristic case, as well as pointing out useful references. In particular, the author thanks Skip Garibaldi for the reference for \cref{lem:conjugateg2s}, and for helping with the proof of \cref{normalityg2}. Lastly, the author is also grateful to the anonymous referees for helpful suggestions.
\subsection*{Assumptions and notations}
In this paper, $k$ will denote a field, $G$ an absolutely simple linear algebraic group and $\g$ its associated Lie algebra. 
We use the same restrictions on the characteristic as \cite{chayet2020class}, i.e.\@ $\kar k \geq h+2$ or $0$ with $h$ the Coxeter number of the Dynkin diagram. In this case, the Weyl modules of highest weight $2\omega_1$ for $G_2$ and highest weight $2\omega_4$ for $F_4$ are irreducible, as are the Weyl modules of weight $2\tilde{\alpha}$, where $\tilde{\alpha}$ is the highest root, by \cite{lub01}. In particular, we can do character computations for these representations independent of the field $k$.

\section{Preliminaries}\label{preliminaries}
\subsection{Representation theory of simple algebraic groups and Lie algebras}
The irreducible representations of a split simple algebraic group over a field $k$ are classified by the dominant weights of the associated root system (\cite[Theorem~22.2]{milne2017algebraic}). To any dominant weight $\lambda$ we can also associate a so-called \emph{Weyl module} $V(\lambda)$.  When the characteristic is zero or large enough, these Weyl modules turn out to be irreducible (as is the case for the Weyl modules we are considering under the characteristic assumptions above).  We will often identify a representation by its associated dominant weight using labelling as in \cite[Plates I-IX, p.264-290]{Bourbaki46}. As mentioned in \cite[\S 7, p.10-11]{chayet2020class}, for not necessarily split simple groups there is a unique representation that becomes isomorphic to $V(\lambda)$ when base changing to $\bar{k}$, so we will denote that representation by the same notation. For further discussion on irreducible representations of simple (or more generally, reductive) groups, see \cite{jantzenrepalggroups}.

	We will work with the representations of algebraic groups as representations of their Lie algebras. (Note that not every representation of the Lie algebra corresponds to a representation of the associated algebraic group.) We will use the same notation $V(\lambda)$ for the Weyl module when considered as a Lie algebra representation. 
	

	We will use the characters of representations of Lie algebras to compute dimensions of certain representations (\cite[Corollary~24.6]{FultonHarris}) and morphism spaces (the argument for finite groups is given in \cite[p.12]{FultonHarris}, but also holds for Lie algebras), as well as decompositions of symmetric powers (\cite[Exercise~23.39]{FultonHarris}).

\begin{remark}\label{reptheoryremark}
    Characters of representations of Lie algebras (in particular the Weyl modules, over algebraically closed fields) can be computed using Sage \cite{sagemath}.
\end{remark}

We will also need the following known result, see e.g.\@ \cite[VIII.\S6.4]{Bourbaki7-9} (for the positive characteristic proof, see the slightly more general \cite[Lemma~2.9]{chayet2020class}). Here the Killing form will be denoted by $K$, and $\langle \  \mid\  \rangle$ is the canonical bilinear form on the weight lattice of the root system as in \cite[\S VI.1.12]{Bourbaki46}.
\begin{proposition}\label{casimirisscalar} Let $\g$ be a simple Lie algebra (associated to an algebraic group $G$) over $k$.
	Let $\pi \colon \g \to \End(V)$ be equivalent to the Weyl module $V(\lambda)$ over an algebraic closure of $k$, with $\lambda$ a dominant weight. Denote by $\delta$ half the sum of the positive roots. Then:
	\begin{enumerate}
		\item $ \sum \pi(X_i)\pi(Y_i) = \langle \lambda \mid \lambda + 2 \delta \rangle \cdot \id_V $ for dual bases $\{X_i\}$ and $\{Y_i\}$ with respect to the Killing form,
		\item for all $X,Y\in \g$ we have
		$ \Tr(\pi(X)\pi(Y)) = \tfrac{\langle \lambda \mid  \lambda + 2 \delta \rangle \cdot \dim V}{\dim \g}K(X,Y)$.
	\end{enumerate}
\end{proposition}	
\subsection{Construction of the algebras}
To construct $\Ag$, we will start from the symmetric square $\Sq^2 \g$ of the Lie algebra. However, the representation $\Sq^2\g$ is too large, and the resulting algebra will not be simple. To resolve this issue, Chayet and Garibaldi introduced a ``projection'' operator $S\colon \Sq^2\g \to \End(\g)$.
\begin{definition}[$\Ag$ as a vector space]
	\leavevmode
	\begin{enumerate}
		\item We define a linear map $S\colon \Sq^2\g \to \End(\g)$ by
		\[S(XY)\coloneqq h^{\vee} \ad X \bullet \ad Y + \tfrac{1}{2}(X K(Y,\mathunderscore) + Y K(X, \mathunderscore) )\text{,} \]
		where $h^{\vee}$ is the dual Coxeter number of the associated root system, and $\bullet$ denotes the usual Jordan product $a\bullet b = \tfrac{1}{2}(ab+ba) $. This is well-defined, since it is symmetric in $X$ and $Y$.
		
		\item We define  $\Ag \coloneqq \im(S)$.
	\end{enumerate}
\end{definition}

To construct an algebra, we of course need to describe a product on the vector space.
\begin{definition}[The product $\diamond$]\label{diamondprod}
	We define a product $\diamond \colon \Ag \times \Ag \to \Ag$ by
	\begin{align*}
		S(AB)\diamond S(CD) \coloneqq \tfrac{h^{\vee}}{2} ( &
		S(A, (\ad C \bullet \ad D)B) + S((\ad C \bullet \ad D)A,B) \\
		&+ S(C, (\ad A \bullet \ad B)D) + S( (\ad A \bullet \ad B)C, D)\\
		&+ S([A,C][B,D])+S([A,D][B,C]) )\\
		+ \tfrac{1}{4}(& K(A,C)S(BD) + K(A,D)S(BC)\\
		&+ K(B,C)S(AD) + K(B,D)S(AC) ) \quad \forall A,B,C,D \in \g\text{,}
	\end{align*}
	and linearly extending this product in both terms.
\end{definition}
There is a bilinear form $\tau$ on these algebras that \emph{associates} with the product defined above.

An element $a$ of the algebra $A(\g)$ is an endomorphism of the Lie algebra $\g$. Thus the algebra comes equipped with a trace form. We denote $\varepsilon(a) \coloneqq \tfrac{1}{\dim \g}\Tr(a)$. Note that we can invert $\dim \g = \rk (\g) (h+1)$, since $\rk(\g),h+1<h+2$.
\begin{definition}[The bilinear form $\tau$]\label{taudef}
	Define the bilinear form $\tau \colon \Ag \times \Ag \to k$ by
	\[ \tau(a,a') \coloneqq \varepsilon(a\diamond a')\text{.} \]
\end{definition}

\begin{lemma}[{\cite[Lemma~6.1]{chayet2020class}}]\label{associativitytau}
	The bilinear form $\tau$ on $\Ag$ associates with $\diamond$, i.e.\ for all $ a,a',a'' \in \Ag$,
	\[ \tau(a\diamond b,c) = \tau(a, b\diamond c) . \]
\end{lemma}
In \cite{chayet2020class}, the module structure of $A(\g)$ was completely identified as well.
\begin{proposition}[{\cite[Proposition 7.2]{chayet2020class}}]\label{modulestructure}
	Let $\g$ be a Lie algebra associated to an absolutely simple algebraic group $G$ appearing in \cref{tablemodulestructure} and assume that $\kar k =0$ or $\kar k \geq h+2 $. As a representation of $G$, we have $\Ag = k \oplus V(\lambda)$, where $\lambda$ is as in \cref{tablemodulestructure}, and the Weyl module $V(\lambda)$ is irreducible.
\end{proposition}

\begin{table}[h]
	\centering
	\begin{tabular}{c|c c c c c c}
		type of $G$& $A_2$ & $G_2$ & $F_4$ & $E_6$ & $E_7$&$E_8$ \\
		\hline
		dual Coxeter number $h^{\vee}$ & $3$&$4$&$9$&$12$&$18$&$30$ \\
		Coxeter number $h$ & $3$&$6$&$12$&$12$&$18$&$30$ \\
		Dominant weight $\lambda$ & $\omega_1 + \omega_2$&$2\omega_1$&$2\omega_4$&$\omega_1+\omega_6$&$\omega_6$&$\omega_1$ \\
		Dimension of $V(\lambda)$ & $8$&$27$&$324$&$650$&$1539$&$3875$ \\
	\end{tabular}
	\medskip
	
	\caption{The table from \cite{chayet2020class}. The fundamental dominant weights are labelled using Bourbaki labelling.}\label{tablemodulestructure}
\end{table}

In the last section of \cite{chayet2020class}, the authors described an embedding of these algebras into the endomorphism ring of a natural representation for the groups of type $A_2,G_2,F_4,E_6$ and $E_7$. It is this embedding that we will use to obtain new descriptions for the algebras in question.
\begin{proposition}[{\cite[Proposition~10.5]{chayet2020class}}]\label{embedding}
	If $G$ has type $A_2,G_2,F_4,E_6$ or $E_7$ and $\pi\colon G \to \GL(W)$ is the natural irreducible representation of dimension $3,7,26,27$ or $56$ respectively, then the formula
	\begin{equation}\label{sigma}
		\sigma(S(XY)) =6h^{\vee} \pi(X)\bullet \pi(Y) - \tfrac{1}{2}K(X,Y)
	\end{equation}
	defines an injective $G$-equivariant linear map
	\[\sigma \colon A(\g) \hookrightarrow \End(W)\text{.} \]
	Moreover, $\sigma$ maps the identity $\id_\g$ to the identity $\id_W$.	
\end{proposition}

In the following subsections, we introduce the irreducible representations for types $G_2$ and $F_4$ to the unfamiliar reader.
\subsection{The octonion algebra}\label{octonionsubsection}

Though octonion algebras can be defined over any characteristic, we restrict ourselves to $\kar k \neq 2$. In this case, it is a well-known fact that the $7$-dimensional representation of $G_2$ arises naturally from the theory of \emph{composition algebras}, and more precisely the octonion algebras. That is why we will outline some results about these objects that will be of use later, to determine what the constructed algebra of type $G_2$ looks like. 

The treatment given in this section is based on \cite{SpringerVeldkamp}.
\begin{definition}[Octonion algebra]\leavevmode
	\begin{enumerate}
		\item An \emph{octonion algebra} is an $8$-dimensional $k$-algebra $A$ equipped with a nondegenerate quadratic form $N\colon A \to k$ such that
	\[N(a b) = N(a)N(b),\]
	for all $a,b\in A$. We call $N$ the \emph{norm} of the composition algebra. We denote its associated bilinear form by $\langle \cdot, \cdot \rangle$, so $\langle x,y \rangle = N(x+y)-N(x) - N(y)$ for any two octonions $x,y \in A$. We will say two octonions $a,b \in A$ are orthogonal and write $a\perp b$ whenever $\langle a,b \rangle =0$.
		\item Let $A$ be an octonion algebra with identity $e$.
	We define the \emph{standard involution} $\overline{\cdot}\colon A \to A$ by
	\[ \overline{x} = \langle x,e \rangle e -x, \] for all $x\in A$. The standard involution is an anti-automorphism of the octonion algebra.
	\end{enumerate}
\end{definition}
To reach our goal of an alternate description of $A(\g_2)$, we will need certain identities satisfied by octonion algebras. We have collected them in the proposition below.

\begin{proposition}\label{octonionidentities} For $x,y,z \in A$ we have the following identities:
	\begin{enumerate}\itemsep1ex
		\item\label{minimimalpoly}
		$xy+yx - \langle x,e\rangle y - \langle y,e\rangle x +\langle x, y \rangle e =0 $, 

		\item\label{1.2.6} 
	$		\langle xy,z \rangle = \langle y, \overline{x}z \rangle\text{, }
		\langle xy,z \rangle = \langle x, z\overline{y} \rangle\text{, }
		\langle xy,\overline{z} \rangle = \langle yz, \overline{x}\rangle\text{,}
	$
		\item\label{normrule}
	$
		x(\overline{x}y) = N(x)y\text{, }
		(x\overline{y})y = N(y)x\text{.}
	$
		
 		\item\label{moufangidentities} (Moufang identities): 
	$
		(zx)(yz) = z((xy)z)\text{, }
		z(x(zy)) = (z(xz))y\text{, }
		x(z(yz)) = ((xz)y)z\text{.}
	$
	\end{enumerate}
\end{proposition}
\begin{proof}
	These are \cite[Proposition~1.2.3, Lemma~1.3.2, Lemma~1.3.3 and Proposition~1.4.1]{SpringerVeldkamp}, respectively.
\end{proof}
The octonion algebras are examples of so-called \emph{alternative algebras}, i.e. the associator $\{x,y,z\} \coloneqq (xy)z-x(yz)$ satisfies the property
\[ \{x_{\pi(1)},x_{\pi(2)},x_{\pi(3)}  \} = \sgn(\pi) \{x_1,x_2,x_3\}\]
for every permutation $\pi \in \mathcal{S}_3$.

It will be convenient to make use of a standard basis. 
\begin{proposition}\label{orthogonalanisotropic}
	Any octonion algebra $A$ over a field $k$ with $\kar k \neq 2$ has an orthogonal basis of the form $e,e_1 = a,e_2 = b,e_3 = ab, e_4 = c,e_5=ac,e_6 = bc,e_7 = (ab)c$, with $N(a)N(b)N(c)\neq0$.
\end{proposition}
\begin{proof}
	See \cite[Corollary~1.6.3]{SpringerVeldkamp}.
\end{proof}
\begin{remark}\label{standardbasis}
	In case $k$ is algebraically closed, we can assume $N(a)=N(b)=N(c)=1$, and we call such a basis a \emph{standard basis} for the octonions. For a basis of this form, the multiplication is encoded by the Fano plane (see \cref{fanoplane}).
\end{remark}

\begin{figure}[h]
	\centering
	\begin{tikzpicture}[scale=2]
	\begin{pgfonlayer}{nodelayer}
		\node [style=new style 0] (0) at (0, 0.577) {$e_4$};
		\node [style=new style 0] (1) at (0, 1.73) {$e_7$};
		\node [style=new style 0] (2) at (1, 0) {$e_5$};
		\node [style=new style 0] (3) at (-1, 0) {$e_6$};
		\node [style=new style 0] (4) at (0, 0) {$e_3$};
		\node [style=new style 0] (5) at (0.5, 0.866) {$e_2$};
		\node [style=new style 0] (6) at (-0.5, 0.866) {$e_1$};
	\end{pgfonlayer}
	\begin{pgfonlayer}{edgelayer}
		\draw [style=new edge style 1] (3) to (6);
		\draw [style=new edge style 1] (6) to (1);
		\draw [style=new edge style 1] (1) to (5);
		\draw [style=new edge style 1] (5) to (2);
		\draw [style=new edge style 1] (2) to (4);
		\draw [style=new edge style 1] (4) to (3);
		\draw [style=new edge style 1] (4) to (0);
		\draw [style=new edge style 1] (0) to (1);
		\draw [style=new edge style 1, bend left=45] (4) to (6);
		\draw [style=new edge style 1, bend left=45] (6) to (5);
		\draw [style=new edge style 1, bend left=45] (5) to (4);
		\draw [style=new edge style 1] (5) to (0);
		\draw [style=new edge style 1] (0) to (3);
		\draw [style=new edge style 1] (6) to (0);
		\draw [style=new edge style 1] (0) to (2);
	\end{pgfonlayer}
\end{tikzpicture}
	\caption{The Fano plane mnemonic. If one follows the arrows when multiplying, then the outcome is equal to the third point on the line. Otherwise it is equal to minus the third point on the line, e.g.\@ $e_6e_2 = e_4$.}\label{fanoplane}
\end{figure}
From the viewpoint of the algebraic group of type $G_2$, it is more natural to work with the $7$-dimensional irreducible representation of the \emph{pure} octonions $W = e^{\perp}$. To do this, we need to modify the octonion multiplication to a multiplication on $W$.
\begin{definition}
	We will define the \emph{Malcev product} on the octonions by
	\[ a*b \coloneqq a b - b a \quad \text{for all } a,b\in A\text{.} \]
	If $a,b\in W$, then we have $a*b = 2ab + \langle a,b\rangle e \in W$.
\end{definition}
\begin{remark}
	The product $*$ is called the Malcev product because it turns the pure octonions into a \emph{Malcev algebra}, see e.g.\@ \cite{myung2013malcev}.
\end{remark}

The product $*$ is anticommutative. It is easy to see that \cref{octonionidentities}\ref{1.2.6} extends to the Malcev product, in the following way.
\begin{lemma}\label{MalcevForm}
	For $x,y,z \in W$ pure octonions, we have
	\[ \langle x*y,z \rangle  =  \langle x,y*z \rangle\text{.} \]
\end{lemma}
The following result will simplify many of the computations in the next section. The proof is due to Tom De Medts.
\begin{lemma}\label{technicallem} \leavevmode
	\begin{enumerate}
		\item\label{lem:technicallem:item1} If $a,b,x\in W$ are pure octonions, then \[\{a,x,b\}= \tfrac{1}{2}x*(a*b) + \langle a,x\rangle b - \langle b,x\rangle a\text{.} \]
		
		\item\label{lem:technicallem:item2} If $a,b,x\in W$ are pure octonions, then
		\[ (x*a)*b + (x*b)*a = 2\langle a,x\rangle b + 2\langle b,x\rangle a -4\langle a,b\rangle x\text{.}\]
	\end{enumerate}
\end{lemma}
\begin{proof}
	\begin{enumerate}
		\item Using \cref{octonionidentities}\ref{minimimalpoly} we get for two pure octonions $a,b\in W$ that \linebreak $ab+ba = - \langle a,b\rangle e $. With this identity, we compute:
		\begin{align*}
			\{a,x,b\} &= (ax)b -a(xb) \\
			&= (-xa-\langle a,x\rangle e)b - a(-bx-\langle b,x\rangle e)\\
			&= 2\{a,x,b\} - x(ab) + (ab)x - \langle a,x\rangle b + \langle b,x\rangle a\text{.}
		\end{align*}		
		Reordering the terms gives $x*(ab) = \{a,x,b\} - \langle a,x\rangle b + \langle b,x\rangle a $.
		
		Now note that $x*(a*b) = x*(2ab + \langle a,b\rangle e) = 2x*(ab)$. Using
this, we get the formula in \ref{lem:technicallem:item1}.
		\item The second item follows from the first by applying it to both $\{x,a,b\}$
 and $\{x,b,a\}$, then summing the equations and noting the left hand side is zero by the alternativity of the octonions. \qedhere
 	\end{enumerate}
\end{proof}
\subsection{The Albert algebras}\label{albertsubsection} We will only consider the Albert algebras over fields $k$ with $\kar k\neq 2,3$.
	In this case, the $26$-dimensional representation for type $F_4$ can be constructed as matrices over the octonions. We give a short overview in this subsection. 
	
	We regard the split Albert algebra as the hermitian matrices $\mathcal{H}_3{(\mathbb{O})}$, where $\Oct$ denotes the split octonions. We write
\[ \mathbf{1}_1 \coloneqq \begin{bsmallmatrix}
	1&0&0\\
	0&0&0\\
	0&0&0
\end{bsmallmatrix},\mathbf{1}_2 \coloneqq \begin{bsmallmatrix}
	0&0&0\\
	0&1&0\\
	0&0&0
\end{bsmallmatrix},\mathbf{1}_3 \coloneqq \begin{bsmallmatrix}
	0&0&0\\
	0&0&0\\
	0&0&1
\end{bsmallmatrix}, \]
and for octonions $a,b,c\in \Oct$
\[ a_1 \coloneqq \begin{bsmallmatrix}
	0&0&0\\
	0&0&a\\
	0&\bar{a}&0
\end{bsmallmatrix}, b_2 \coloneqq \begin{bsmallmatrix}
	0&0&\bar{b}\\
	0&0&0\\
	b&0&0
\end{bsmallmatrix},c_3 \coloneqq \begin{bsmallmatrix}
	0&c&0\\
	\bar{c}&0&0\\
	0&0&0
\end{bsmallmatrix}. \]	
	A generic element of $\mathcal{H}_3{(\mathbb{O})}$ is then of the form
\[  
\begin{bsmallmatrix}
	\alpha_1 & c & \bar{b} \\
	\bar{c} & \alpha_2 & a \\
	b & \bar{a} & \alpha_3
\end{bsmallmatrix}
= \alpha_1\mathbf{1}_1+\alpha_2\mathbf{1}_2+\alpha_3\mathbf{1}_3+a_1+b_2+c_3\text{,}
\]

with $\alpha_1,\alpha_2,\alpha_3\in k$ and $a,b,c\in \mathbb{O}$. To keep the notation uniform with the octonion algebras, the unit of the algebra $\mathbf{1}_1+\mathbf{1}_2+\mathbf{1}_3$ will be denoted $e$.

\begin{definition}\leavevmode
	\begin{enumerate}
		\item We define the split Albert algebra to be the hermitian matrices over the split octonions $\mathcal{H}_3(\mathbb{O})$, equipped with the Jordan product, i.e. for $a,b\in \mathcal{H}_3(\Oct)$:
			\[ a\cdot b \coloneqq \tfrac{ab+ba}{2}\text{,} \]
			where the multiplication on the right-hand side is the usual matrix multiplication.
		
		\item We define a nondegenerate bilinear form $\langle \cdot,\cdot \rangle\colon \mathcal{H}_3(\mathbb{O}) \times \mathcal{H}_3(\mathbb{O}) \to k $ by the formula
		\[\langle x,y \rangle  \coloneqq \Tr(xy)\text{.}\]
		We will sometimes refer to this bilinear form as the \emph{trace form}.
		
		\item Over an arbitrary field $k$, an Albert algebra is a $k$-algebra equipped with a bilinear form $\langle \cdot,\cdot\rangle_A $ and a trilinear form $\langle \cdot,\cdot,\cdot \rangle_A$ such that there is an isomorphism of algebras $A\otimes_k\overline{k}\cong \mathcal{H}_3(\mathbb{O})$ sending the extension of the bilinear form $ \langle \cdot,\cdot\rangle_A $ to $\langle \cdot,\cdot \rangle$ and the extension of the trilinear form $\langle \cdot,\cdot,\cdot \rangle_A$ to the trilinearisation $\langle \cdot,\cdot,\cdot \rangle$ of the usual determinant form $\det$ on the $3\times3$ matrices over $\Oct$.
		
			We will omit writing the index $A$ throughout this article for convenience.
	\end{enumerate}	
\end{definition}
\begin{remark}
	We will write both the trace form on the Albert algebra and the bilinear form on the octonions as $\langle\cdot,\cdot\rangle$. This should not cause any confusion however, as for two octonions $a,b\in \mathbb{O}$ and $i,j\in \{1,2,3\}$ we have $\langle a_i,b_j\rangle = \delta_{i,j}\langle a,b\rangle$, where the $\langle \cdot ,\cdot \rangle$ in the left hand side and right hand side are bilinear forms in different algebras.
\end{remark}
We immediately have some routine computations we will use regularly.
\begin{lemma} \label{albertmultiplication}The following multiplication rules hold in the Albert algebra over an algebraically closed field $k$.
	\begin{enumerate}
		\item $\mathbf{1}_i^2 = \mathbf{1}_i$ for $i\in \{1,2,3\}$,
		\item $\mathbf{1}_i\cdot \mathbf{1}_j = 0$ for $i\neq j$ and $i,j\in \{1,2,3\}$,
		\item $\mathbf{1}_i \cdot a_i = 0$ for $i\in \{1,2,3\}$ and $a\in \mathbb{O}$,
		\item $\mathbf{1}_i \cdot a_j = \tfrac{1}{2}a_j$ for $i,j\in \{1,2,3\}$ with $i\neq j$ and $a\in \Oct$,
		\item $a_i\cdot b_i$ = $\tfrac{1}{2}\langle a,b\rangle(\mathbf{1}_j+\mathbf{1}_k)$ for $\{i,j,k\} = \{1,2,3\}$ and $a,b\in \Oct$,
		\item $a_i\cdot b_j = \tfrac{1}{2}(\overline{ab})_k$ for $i,j,k$ a cyclic permutation of $1,2,3$.
	\end{enumerate}
\end{lemma}
\begin{proof}
	These are precisely \cite[Identities~(4.26)-(4.31)]{nonassocalg} adapted to our notation.
\end{proof}
The following lemma will be essential for our treatment of the $F_4$ case.
\begin{lemma}\label{equation512}
	For $x,y,z \in A$ with $A$ an Albert algebra over a field $k$, we have 
	\begin{multline*}
		x\cdot(y\cdot z)+y\cdot (x\cdot z)+z\cdot (x\cdot y) = \\\langle x,e \rangle y\cdot z +\langle y,e \rangle x\cdot z + \langle z,e \rangle x\cdot y 
		+\tfrac{1}{2}( \langle x,y\rangle -\langle x,e\rangle \langle y,e\rangle)z\\+\tfrac{1}{2}( \langle x,z\rangle -\langle x,e\rangle \langle z,e\rangle)y 
		+\tfrac{1}{2}( \langle y,z\rangle -\langle y,e\rangle \langle z,e\rangle)x+ 3\langle x,y,z\rangle e\text{.}
	\end{multline*}
\end{lemma}
\begin{proof}
	This is \cite[Equation~5.12]{SpringerVeldkamp}.
\end{proof}

\subsection{Standard derivations}\label{subsecstandardder}

In this section, we will use so-called \emph{standard derivations}, introduced and studied by R.\@~D.\@~Schafer both for the octonion algebras and the Albert algebras in \cite{nonassocalg}, to make sense of the embedding given in \cite{chayet2020class}. They are defined over arbitrary fields $k$.

The reason we look at the derivations of these algebras is because the derivation algebra of an octonion algebra (respectively an Albert algebra) is equal to the Lie algebra of the automorphism group of type $G_2$ (respectively $F_4$) by \cite[Proposition~2.4.5 and Corollary~7.2.2]{SpringerVeldkamp}.

Unlike Schafer, we will denote composition of operators from right to left, i.e.\@ $\phi\psi(x) = \phi(\psi(x))$ for two operators $\phi,\psi\colon U \to U$ on a vector space $U$ and $x\in U$. Note that due to this different convention, the definition of $D_{a,b}$ differs by a minus sign.
\begin{definition}\leavevmode
	\begin{enumerate}
		\item Let $a,b \in A$ be octonions. We define the standard derivation $D_{a,b}$ by
		\[ D_{a,b} \coloneqq  [L_a,L_b] + [L_a,R_b] + [R_a,R_b]\text{,} \]
		where $L_a$ (respectively $R_a$) stands for left (respectively right) multiplication by $a$ for all $a\in A$.
		\item 	Let $a,b\in A$, where $A$ is an Albert algebra. We define the \emph{standard derivation}  $D_{a,b}$ by
		\[ D_{a,b} \coloneqq \left[L_a,L_b\right]\text{,} \]
		where $L_a$ stands for left multiplication by $a$.
	\end{enumerate}
\end{definition}

\begin{lemma} \label{derivationlemma}Let $a,b\in A$ where $A$ is either an octonion algebra or an Albert algebra with identity $e$. The following hold:
	\begin{enumerate}
		\item $D_{a,b}$ is a derivation of $A$,
		\item $D_{a,b} = -D_{b,a}$,
		\item $D_{a,a}=0$,
		\item\label{item3} $D_{e,b} = 0$,
		\item $D_{a,b}$ is skew-symmetric with respect to $\langle \cdot,\cdot \rangle$, i.e.
		\[\langle D_{a,b}(u),v\rangle = -\langle u, D_{a,b}(v)\rangle\text{.}\]
	\end{enumerate}
\end{lemma}
\begin{proof}
	\begin{enumerate}
		\item
		This is proven in \cite[Identity (3.70), p.77]{nonassocalg} and \cite[Identity (4.6), p.92]{nonassocalg}.
		\item For $A$ an octonion algebra, this follows immediately from the fact that octonion algebras are alternative. For $A$ an Albert algebra, it is even easier.
		\item All commutators in the definitions become zero.
		\item All commutators in the definitions become zero.
		\item We prove this for $A$ an octonion algebra. The Albert algebra case is analogous but easier.
		We get
		\begin{multline*}
				\langle D_{a,b}(u),v\rangle\\
				= \langle a\cdot(b\cdot u) - b\cdot(a\cdot u),v\rangle
				 + \langle a\cdot(u\cdot b) - (a\cdot u)\cdot b , v\rangle
				+ \langle (u\cdot b)\cdot a - (u\cdot a)\cdot b,v \rangle \\
				= \langle b\cdot u,\overline{a}\cdot v\rangle - \langle a\cdot u, \overline{b}\cdot v \rangle + \langle u\cdot b,\overline{a}\cdot v\rangle -\langle a\cdot u , v\cdot \overline{b}\rangle + \langle u\cdot b, v\cdot \overline{a} \rangle - \langle u\cdot a,v \cdot \overline{b} \rangle \\
				= \langle  u,\overline{b}\cdot(\overline{a}\cdot v)\rangle - \langle u, \overline{a}\cdot(\overline{b}\cdot v )\rangle + \langle u,(\overline{a}\cdot v)\cdot \overline{b}\rangle -\langle u , \overline{a}\cdot (v\cdot \overline{b})\rangle + \langle u, (v\cdot \overline{a})\cdot \overline{b} \rangle - \langle u,(v \cdot \overline{b})\cdot \overline{a} \rangle \\
				=  \langle u , - D_{\overline{a},\overline{b}}(v) \rangle\text{.}
		\end{multline*}
		Now by \ref{item3} and the definition of the involution, we have $D_{\overline{a},\overline{b}} = D_{a,b}$, ending our proof. \qedhere
	\end{enumerate}
\end{proof}
\begin{remark}
	As an anonymous referee correctly points out, we can prove the first item directly by the formulas in this article. By computing characters, we know that for the space of traceless octonions $W$, the exterior square decomposes as $\wedge^2 W = \g_2 \oplus \{ R_c^* \mid c\in W \}$, where $R^*_x\in \End(W)$ is the operator sending a pure octonion $a$ to $a*x$, i.e.\@ $R^*_x$ is right Malcev multiplication by $x$. These two subspaces are orthogonal by the usual trace form, as this is also an equivariant bilinear form. Thus, using \cref{technicallem}~\ref{lem:technicallem:item1} and \ref{lem:technicallem:item2} we can prove that $\Tr(R_c^*\bullet D_{a,b}) = 0 $ for all $a,b,c\in W$, and thus that all $D_{a,b}$ are in fact derivations.
\end{remark}
These standard derivations span the entire derivation space.
\begin{theorem}[Schafer]
	Every derivation of an octonion algebra or an Albert algebra over a field $k$ with $\kar k \neq 2,3$ is a linear combination of standard derivations.
\end{theorem}
\begin{proof}
	See \cite[Corollary 3.29]{nonassocalg} and \cite[Corollary~4.10]{nonassocalg}.
\end{proof}
We can easily describe the commutator product of standard derivations.
\begin{lemma}\label{liebracket} Let $D$ be a derivation of $A$, and $a,b \in A$ with $A$ an octonion or an Albert algebra. Then we have
	\[ [D,D_{a,b}] = D_{Da,b} + D_{a,Db} \text{.} \]
\end{lemma}
\begin{proof}
	These are \cite[Identities~(3.72) and (4.62)]{nonassocalg}.
\end{proof}
We end this section with the following remark.
\begin{proposition}
	Let $\sigma$ denote the embedding from \cref{embedding} for $G= G_2$ or $G=F_4$. Then $\im(\sigma)$ is contained in the subspace $\mathcal{S}(W)$  of symmetric operators with respect to $\langle \cdot, \cdot \rangle$. 
\end{proposition}
\begin{proof}
	By the last item of \cref{derivationlemma}, any standard derivation of the algebra representation is skew-symmetric with respect to the bilinear form. Since any derivation is a linear combination of standard derivations, any derivation is skew-symmetric. This, combined with the defining equation \eqref{sigma} of $\sigma$ tells us that $ \im(\sigma)$ is contained in $\mathcal{S}(W)$.
\end{proof}
\begin{remark}
	This can be seen more abstractly, as both $G_2$ and $F_4$ stabilise a nondegenerate bilinear form on the representation of $2\omega_1$ and $2\omega_4$ respectively, and are thus contained in a special orthogonal group. But the Lie algebra of the special orthogonal group is precisely the Lie algebra of skew-symmetric matrices, thus the Lie algebra of derivations is contained in the skew-symmetric matrices.
\end{remark}

\subsection{Symmetric operators and the symmetric square}\label{symmsquaresec}

In this short section, let $W$ be an arbitrary finite-dimensional vector space equipped with a nondegenerate bilinear form $\langle\cdot,\cdot\rangle$. 

We denote the space of symmetric operators on $W$ with respect to the associated bilinear form with $\mathcal{S}(W)$, and the symmetric square $\Sq^2W \coloneqq \faktor{W\otimes W}{(a\otimes b - b\otimes a \mid  a,b\in W)} $.
\begin{lemma}\label{symmoperatorssymsquare}
	We have an isomorphism
	\begin{align*}
		\varphi\colon\Sq^2W &\to \mathcal{S}(W) \\
		ab &\mapsto \tfrac{a\langle b,\cdot\rangle+b\langle a, \cdot\rangle}{2}  \text{.}
	\end{align*}
	Moreover, for a group $G$ acting on $W$ and stabilising the bilinear form, this isomorphism is $G$-equivariant.
\end{lemma}
\begin{proof}
	It suffices to observe that for any $\omega \in G$, due to the $G$-invariance of the bilinear form,
	\begin{equation*}
		\tfrac{\omega a\langle \omega b,\cdot\rangle+\omega b\langle \omega a, \cdot\rangle}{2} = \omega \circ \left(\tfrac{ a\langle b,\cdot\rangle+ b\langle  a, \cdot\rangle}{2} \right)\circ \omega^{-1} \text{.}
	\end{equation*}
	By the nondegeneracy of the bilinear form, $\varphi$ is injective. By dimension count, it is also an isomorphism.
\end{proof}
As an anonymous referee points out, this is the identification $\Sq^2 W \hookrightarrow \Sq^2 W \oplus \wedge^2 W \cong W \otimes W \cong \End(W)$, where the last identification is given by the inner product.
\begin{notation}\label{rem:shorthand}
	This isomorphism will play a key role in the following. For simplicity, we will write $ab$ as a shorthand for the operator $\varphi(ab) =\tfrac{ a\langle b,\cdot\rangle+ b\langle  a, \cdot\rangle}{2} $. 
	\end{notation}

In Equation \eqref{sigma}, the Killing form of two derivations occurs, which we can regard as a multiple of the trace form by \cref{casimirisscalar}. The following lemma will be useful to compute the embedding $\sigma$; see \cref{FORMULAS} below.
\begin{lemma}\label{traceformoctonions}
	With the notation from above, we have
	\[\Tr(ab) = \langle a, b\rangle \text{.} \]
\end{lemma}
\begin{proof}
	Extend $\{a,b\}$ to a basis $B$ of $W$ and write $ab$ as a matrix with respect to the basis $B$. Then the only diagonal elements that are non-zero appear in the columns associated to $a$ and $b$. These non-zero entries are both equal to $\tfrac{1}{2}\langle a,b \rangle$, so the trace of the matrix is equal to $\langle a,b\rangle$.
\end{proof}

Given the isomorphism from \cref{symmoperatorssymsquare} and an orthonormal basis, we will use two constructions in the following sections.
\begin{lemma}\label{lem:constructionssymmsquare}
	Let $\{b_1,\dots,b_n \}$ be an orthonormal basis for $(W,\langle \cdot,\cdot \rangle)$. Then (using the identification in \cref{rem:shorthand})
	\begin{enumerate}
		\item\label{item1:lem:constructionssymsquare} $\id_W = \sum_i b_ib_i$,
		\item\label{item2:lem:constructionssymsquare} The elements $(b_1+b_i)(b_1-b_i), b_ib_j$ for $i\in \{2,\dots,n\}, i\neq j\in \{1,\dots,n \}$ form a basis for the space of traceless symmetric operators.
	\end{enumerate} 
\end{lemma}
\begin{proof}
	For the first item, let $v=\sum_k \lambda_kb_k \in W$ be arbitrary. Then we have
	$ \left( \sum_i b_ib_i \right)(v) = \sum_k \lambda_k\left( \sum_i b_ib_i \right)(b_k) = \sum_k \lambda_kb_k = v$. Since $v$ was arbitrary, we have $\id_W = \sum_i b_ib_i$.
	
	For the second item, we can use \cref{traceformoctonions} to see that the proposed elements are all traceless symmetric operators. They are clearly linearly independent (using for example the isomorphism from \cref{symmoperatorssymsquare}), thus they span the space of traceless symmetric operators, since its dimension is equal to $\tfrac{n(n+1)}{2}-1=\tfrac{(n+2)(n-1)}{2}$.
\end{proof}

\section{The algebra of type $G_2$}\label{secg2}
Recall that
 the assumption on the characteristic in this section is $\kar k =0$ or $\kar k>7$.
In this part, $W$ denotes the $7$-dimensional $G_2$-representation formed by the pure octonions, and $V$ denotes the $27$-dimensional representation with highest weight $2\omega_1$. In this setting, the embedding from \cref{embedding} becomes, using \cref{casimirisscalar},
\begin{equation}\label{embeddingg2}
	\sigma(S(XY)) = 24X\bullet Y - 2\Tr(XY)\id_{W}\text{.} 
\end{equation}
In this formula, we do not write the embedding $\pi$ of the Lie algebra explicitly.

This part is dedicated to studying the embedding $\sigma$ more deeply, using the octonions. More specifically, we will determine a formula for $\sigma$ in terms of standard derivations, and the bilinear form on the octonions. In this way, we find an alternate description for the algebra $A(\g_2)$. Our main result is \cref{identificationag2}.

\subsection{A formula for standard derivations and their Jordan products}

To find a formula for the embedding $\sigma\colon A(\mathfrak{g}_2) \hookrightarrow \End(W)$, we try to give a formula for the derivations in terms of the bilinear form on the octonions and the Malcev product; see \cref{standardderinprod}.

The proof of the following proposition is due to Tom De Medts.
\begin{proposition}\label{standardderinprod} For $a,b,x\in W$ pure octonions, we have
	\[D_{a,b}(x) = 3\langle a,x \rangle b - 3\langle b,x \rangle a +\tfrac{x*(a*b)}{2} \text{.} \]
\end{proposition}
\begin{proof}
	We will prove this using \cref{technicallem}\ref{lem:technicallem:item1} and \cref{octonionidentities}\ref{minimimalpoly}. Using these equations, the formula we have to prove is equivalent to
	\[ D_{a,b}(x) = \{a,x,b \} + 2\langle a,x \rangle b - 2\langle b,x \rangle a\text{.}\]
	Now, using the definition of $D_{a,b}(x)$, this is equivalent to proving
	\begin{align*}
		&a(bx)-b(ax)+ a(xb)- (ax)b+ (xb)a-(xa)b = (ax)b -a(xb) + 2\langle a,x \rangle b - 2\langle b,x \rangle a\\
		\iff & -\langle x,b\rangle a -b(ax) + \langle x,a\rangle b  + (xb)a = (ax)b -a(xb) + 2\langle a,x \rangle b - 2\langle b,x \rangle a\\
		\iff &-b(ax)+ (xb)a = (ax)b -a(xb) -(ax+xa)b +a(bx+xb) \\
		\iff &(xb)a +(xa)b = a(bx)+b(ax) .
	\end{align*} 
	The last equality in this string of equivalences holds, by linearising the equality $c(cx) = -N(c)x = (xc)c$ for any pure octonion $c$ and any octonion $x$. 
\end{proof}

\begin{notation}
	By $R^*_x\in \End(W)$, with $x\in W$ a pure octonion, we denote the operator sending a pure octonion $a$ to $a*x$, i.e.\@ $R^*_x$ is right Malcev multiplication by $x$.
	Using this notation, we can rephrase the previous corollary into the equality
	\[ D_{a,b} = 3\langle a,\cdot \rangle b- 3\langle b,\cdot \rangle a+\tfrac{1}{2}R^*_{a*b}\text{.} \]
\end{notation}
Using \cref{standardderinprod,technicallem}, we can give an explicit formula for the embeddding in \cref{embedding}.
\begin{proposition} \label{FORMULAS}
	If $a,b,c,d\in W$ are pure octonions, then for the image of $S(D_{a,b}D_{c,d})$ under the embedding $\sigma$, we find
	\begin{align*}
		\sigma( S(D_{a,b} D_{c,d}) ) &= -216\left( \langle a,c \rangle bd -\langle a,d \rangle bc + \langle b,d \rangle ac - \langle b,c \rangle ad \right)\\
		&\hspace{3ex}- 36\left( (a*(c*d))b - (b*(c*d))a + (c*(a*b))d-(d*(a*b))c \right) \\
		&\hspace{3ex}+12(a*b)(c*d)  \\
		&\hspace{3ex}+18\left(2\langle a,c \rangle \langle b,d \rangle - 2\langle a,d \rangle \langle b,c \rangle - \langle a*b,c*d \rangle \right)\id_W\text{.}
	\end{align*}
\end{proposition}
\begin{proof}
	We compute $D_{a,b}\bullet D_{c,d}$ using \cref{standardderinprod}. To reduce the amount of terms involved, we first compute $D_{a,b}\circ D_{c,d}$ and get, making use of \cref{MalcevForm},
	\begin{align*}
		D_{a,b}\circ D_{c,d} &= \left(3b\langle a,\cdot \rangle  - 3a\langle b,\cdot \rangle +\tfrac{1}{2}R^*_{a*b} \right)\left(3d\langle c,\cdot \rangle  - 3c\langle d,\cdot \rangle +\tfrac{1}{2}R^*_{c*d}\right) \\
		&=9b\langle a,d\rangle \langle c,\cdot \rangle - 9b\langle a,c\rangle \langle d,\cdot \rangle -9a\langle b,d\rangle \langle c,\cdot \rangle+ 9a\langle b,c\rangle \langle d,\cdot \rangle \\
		&\hspace{3ex}-\tfrac{3}{2}b\langle a*(c*d),\cdot \rangle  + \tfrac{3}{2}a\langle b*(c*d),\cdot \rangle +\tfrac{3}{2}d*(a*b)\langle c,\cdot \rangle  - \tfrac{3}{2}c*(a*b)\langle d,\cdot \rangle \\
		&\hspace{3ex}+\tfrac{1}{4}R^*_{a*b}\circ R^*_{c*d}\text{.}
	\end{align*}
By symmetry, we see that 
	\begin{align*}
		D_{a,b}\bullet D_{c,d} &= -9\left( \langle a,c \rangle bd -\langle a,d \rangle bc + \langle b,d \rangle ac - \langle b,c \rangle ad \right)\\
		&\hspace{3ex}- \tfrac{3}{2}\left( (a*(c*d))b - (b*(c*d))a + (c*(a*b))d-(d*(a*b))c \right) \\
		&\hspace{3ex}+\tfrac{1}{4}R^*_{a*b}\bullet R^*_{c*d}\text{.}
	\end{align*}
	Now we use \cref{technicallem}. Rephrasing the second item of this Lemma into the language of operators and using the notation introduced in \cref{rem:shorthand}, we get
	\[R^*_{a*b}\bullet R^*_{c*d} = 2(a*b)(c*d) - 2\langle a*b,c*d \rangle \id_W\text{.} \]
	To compute the trace of this expression, we use \cref{traceformoctonions,MalcevForm}:
	\begin{align*}
		\Tr(D_{a,b}\bullet D_{c,d} ) &= -9\left( 2\langle a,c \rangle \langle b,d \rangle -2\langle a,d \rangle \langle b,c \rangle \right)-3  \langle a*b,c*d \rangle\\
		&\hspace{3ex}- \tfrac{3}{2}\left( \langle a*(c*d),b \rangle  -\langle b*(c*d),a\rangle + \langle c*(a*b),d\rangle - \langle d*(a*b),c\rangle \right) \\
		&= -18( \langle a,c \rangle \langle b,d \rangle - \langle a,d \rangle \langle b,c \rangle) + 3\langle a*b,c*d \rangle \text{.}	
	\end{align*}
The corollary then follows from plugging these computations into \cref{embeddingg2}.
\end{proof}
For convenience, we also prove a shorter formula in a special case.
\begin{corollary}\label{embeddingsquare}
	Let $a,b,c \in W$ be pure octonions. Then 
	\begin{equation}\label{eq:cor:embeddingsquare}
		\sigma(S(D_{a,b}D_{c,b})) = 12\left( -6\langle a,c\rangle bb -12N(b)ac +6\langle a,b\rangle  cb + 6\langle c,b\rangle ab + (a*b)(c*b) \right) \text{.} 
	\end{equation}
\end{corollary}
\begin{proof}
	We have
	\begin{multline}\label{eq:proof:embeddingsquare}
		\sigma( S(D_{a,b}^2) ) = -216\left( \langle a,a \rangle bb -\langle a,b \rangle ab + \langle b,b \rangle aa - \langle a,b \rangle ab \right)\\
		- 72\left( (a*(a*b))b - (b*(a*b))a\right)+12 (a*b)(a*b)  \\
		+18\left(2\langle a,a \rangle \langle b,b \rangle - 2\langle a,b \rangle \langle a,b \rangle - \langle a*b,a*b \rangle \right)\id_W\text{.} 
	\end{multline}
	We can use \cref{technicallem}\ref{lem:technicallem:item1} to simplify the second line:
	\begin{multline*}
		(a*(a*b))b - (b*(a*b))a + (a*(a*b))b-(b*(a*b))a \\ = -8N(a)bb + 4\langle a,b\rangle ab - 8N(b)aa + 4\langle a,b\rangle ab\text{.}
	\end{multline*}
	For the terms on the fourth line, we get
	\begin{multline*}
		2\langle a,a \rangle \langle b,b \rangle - 2\langle a,b \rangle \langle a,b \rangle - \langle a*b,a*b \rangle  
		\\=2\langle a,a \rangle \langle b,b \rangle - 2\langle a,b \rangle \langle a,b \rangle - \langle 2ab-\langle a,b \rangle e , 2ab-\langle a,b \rangle e \rangle \\
		= 2\langle a,a \rangle \langle b,b \rangle - 2\langle a,b \rangle \langle a,b \rangle -2 \langle a,a \rangle \langle b,b \rangle + 4\langle a,b \rangle \langle a,b \rangle - 2\langle a,b \rangle \langle a,b \rangle = 0 \text{.}
	\end{multline*}
	Plugging these computations into \eqref{eq:proof:embeddingsquare}, we obtain
	\begin{align*}
		\sigma( S(D_{a,b}^2 ) )= 12\Big(-12N(a)bb -12N(b)aa + (a*b)(a*b) + 12\langle a,b\rangle  ab  \Big)\text{.} 
	\end{align*}
	Linearising this equation gives us \eqref{eq:cor:embeddingsquare}.
\end{proof}
As another corollary, we get a nice formula for a preimage of very basic elements.
\begin{corollary}\label{preimage}
	If $a,c\in W$ are orthogonal pure octonions, then we have
	\begin{align*}
		\sigma(S(D_{a,a*c}D_{c,a*c})) = - 768N(a)N(c)ac\text{.}
	\end{align*}
\end{corollary}
\begin{proof}
	We specialise Equation~\eqref{eq:cor:embeddingsquare} to the case where $a\perp c$ and $b=a*c=2ac$ hold. Notice that in this case we also have $a\perp ac \perp c$, because both $a$ and $c$ are pure octonions. Thus we get
	\begin{align*} 
		\sigma(S(D_{a,a*c}D_{c,a*c})) &= 12\big( -12N(2ac)ac + (a*(a*c))(c*(a*c)) \big) \\ &=12(-48N(a)N(c)ac-16N(a)N(c)ca)\text{.}\qedhere
	\end{align*}
\end{proof}
We want to find the multiplication $\star$, defined by
	\[ \sigma(v) \star \sigma(w) \coloneqq \sigma(v\diamond w) \]
	for all $v,w \in A(\g_2)$.
With the formulas derived above, we can compute the multiplication for specific elements in $A(\g_2)$, which we will need in the proof of \cref{identificationag2}.
\begin{proposition}\label{computationg2}
	Let $a,b$ be orthogonal, pure octonions. Then
	\begin{equation}\label{eq:prop:examplecomputation}
		ab\star ab = \tfrac{1}{48}(a*b)(a*b) + \tfrac{1}{12}(\langle b,b\rangle aa+\langle a,a\rangle bb)\text{.}
	\end{equation}
\end{proposition}
\begin{proof}
	First assume $a,b$ are anisotropic. For $X,Y\in \g_2$ we have by \cref{diamondprod}
	\begin{multline}\label{eq:prop:examplecomp:mult}
		S(XY) \diamond S(XY)  \\ =2\left( S(X,(\ad Y \circ \ad X)(Y) ) +S(Y,(\ad X \circ \ad Y)(X) ) - S([X,Y],[X,Y]) \right)   \\
		+ \tfrac{1}{4}K(X,X)S(Y^2) + \tfrac{1}{2}K(X,Y)S(XY)+\tfrac{1}{4}K(Y,Y)S(X^2)\text{.} 
	\end{multline}
	We want to substitute $X=D_{a,a*b},Y=D_{b,a*b}$, so we compute the involved commutator brackets using \cref{liebracket} and \cref{standardderinprod}:
	\begin{align*}
		[D_{a,a*b},D_{b,a*b}] &= D_{b,-16N(a)N(b)a}\text{,} \\
		[D_{a,a*b},D_{b,a}] &=  D_{b,4N(a)a*b} \text{,}\\
		[D_{b,a*b},D_{a,b}] &=  D_{a,4N(b)a*b}\text{.}
	\end{align*}
	Substituting $X= D_{a,a*b}$ and $Y=D_{b,a*b}$ in \eqref{eq:prop:examplecomp:mult}, we get
	\begin{multline}\label{eq2:prop:examplecomp:mult}
		S(XY) \diamond S(XY)  \\ = 128N(a)N(b)^2S(D_{a,a*b}^2)+128N(a)^2N(b)S(D_{b,a*b}^2)
		-512N(a)^2N(b)^2S(D_{a,b}^2)\\+\Tr(D_{a,a*b}^2)S(D_{b,a*b}^2) + \Tr(D_{b,a*b}^2)S(D_{a,a*b}^2)\text{.}
	\end{multline}
	Now we can use \cref{embeddingsquare} to get
	\begin{align*}
		\sigma(S(D_{a,b}^2)) &= 12\left(-12N(a)bb -12N(b)aa + (a*b)(a*b)  \right) \text{,} \\
		\sigma(S(D_{a,a*b}^2)) &= 12\left(-48N(a)N(b)aa -12N(a)(a*b)(a*b) + 16N(a)^2bb \right)\text{,} \\
		\sigma(S(D_{b,a*b}^2)) &= 12\left(-48N(a)N(b)bb -12N(b)(a*b)(a*b) + 16N(b)^2aa \right)\text{.}
	\end{align*}
	We also have, by \cref{FORMULAS},
	\begin{align*}
		\Tr(D_{a,a*b}^2)) &= -192N(a)^2N(b)\text{,} \\
		\Tr(D_{b,a*b}^2)) &= -192N(a)N(b)^2\text{.}
	\end{align*}
	Plugging all of this information in \eqref{eq2:prop:examplecomp:mult}, we get
	\begin{align*}
		768^2N(a)^2N(b)^2 ab\star ab &= 96\cdot 
		4^2\cdot N(a)^2N(b)^2(64(N(b)aa+N(a)bb)+8(a*b)(a*b)) \\
		\iff ab\star ab &= \tfrac{1}{48}(a*b)(a*b) + \tfrac{1}{12}(\langle b,b\rangle aa+\langle a,a\rangle bb)\text{.}
	\end{align*}
	Dividing both sides by $768^2N(a)^2N(b)^2$ gives us \eqref{eq:prop:examplecomputation}. For arbitrary $a,b\in W$, note that we can find a basis of anisotropic elements by \cref{orthogonalanisotropic}. Then the same formula holds for abitrary linear combinations of these anisotropic elements by linearity.
\end{proof}
This is all we will need to determine the product on $\sigma(A(\g_2))$.
\subsection{An isomorphism to the symmetric square of the traceless octonions}
In the case of $G_2$, we obtain a very nice isomorphism for the embedding $\sigma$.
\begin{proposition}\label{symmopiso1}
	The map
	\[\sigma \colon A(\g_2) \to \mathcal{S}(W)\]
	is an isomorphism.
\end{proposition}
\begin{proof}
	Since $\sigma$ is injective and the two underlying vector spaces have the same dimension by \cref{modulestructure}, we conclude $\sigma$ is an isomorphism.
\end{proof}

\subsection{Defining multiplications on $V$}\label{sec:defmultg2}
Recall from \cref{modulestructure} that $A(\g_2) = k\oplus V$, where $V$ is the irreducible $27$-dimensional representation of $G_2$. By \cite[Example~A.6]{chayet2020class}, multiplication is of the form
\begin{equation}\label{eq:multag2}
	(\lambda, u)\star (\mu,v) = (\lambda\mu+f(u,v), \lambda v+\mu u + u\odot v)\text{,} 
\end{equation}
for a certain invariant symmetric bilinear form $f$ and invariant symmetric multiplication $\odot$ on $V$.
Note that under the embedding $\sigma$, $V$ is sent to the subspace of trace zero elements, i.e.\@
\[\sigma(V) = \left\{ \sum_{i} a_ib_i \,\middle\vert\, \sum_{i} \langle a_i,b_i \rangle =0 \right\}\text{.} \]		
\begin{proposition}\label{prop:g2products}
Define for $ab,cd\in \Sq^2W$
\begin{multline}\label{eq:def:odot1g2}
	ab \odot_1  cd \coloneqq  {(a* c)}{(b* d)} + {(a* d)}{(b* c)} + \tfrac{2}{7}(\langle a,c \rangle\langle b,d \rangle +\langle a,d \rangle \langle b,c \rangle)  \id_{W} \text{,} 
\end{multline}
and
\begin{multline}\label{eq:def:odot2g2}
	ab \odot_2  cd  \coloneqq
	\langle a,c\rangle bd + \langle a,d \rangle bc +\langle b,c \rangle ad +\langle b,d \rangle ac \\  - \tfrac{2}{7}(\langle a,c \rangle\langle b,d \rangle +\langle a,d \rangle \langle b,c \rangle) \id_{W}\text{.} 
\end{multline}
Then the space of commutative $G_2$-equivariant products on $\sigma(V)$ is spanned by $\odot_1$ and $\odot_2$, restricted to $\sigma(V)$.
\end{proposition}	
\begin{proof}
	Both products are clearly commutative. The multiplication $\odot_2$ on $\sigma(V)$ is well-defined (meaning that the image of $\odot_2$ is contained in $\sigma(V)$), by \cref{traceformoctonions}. For $\odot_1$ we have the following computation, where in the second equality we use \cref{technicallem} (for $\sum_i a_ib_i, \sum_j c_j d_j \in \sigma(V)$ and suppressing summation):
	\begin{align*}
		\left( \langle a*c,b*d \rangle + \langle a*d,b*c \rangle \right) 
		&= - \langle a, (b*d)*c + (b*c)*d \rangle 
		\\&= - \langle a, 2\langle c,b\rangle d + 2\langle d,b\rangle c -4\langle c,d\rangle b\rangle 
	\\&= -2\langle a,c \rangle\langle b,d \rangle - 2\langle a,d \rangle \langle b,c \rangle+4\langle c,d\rangle \langle a, b\rangle\\
		&= -2\langle a,c \rangle\langle b,d \rangle - 2\langle a,d \rangle \langle b,c \rangle\text{.}
	\end{align*}
	This proves that the right hand side of \cref{eq:def:odot1g2} is again an element of $\sigma(V)$.
		
	The products are $G_2$-equivariant since both $*$ and $\langle \cdot,\cdot \rangle$ are and $\id_W$ is $G_2$-invariant.
		
	It can easily be seen that these two multiplications are linearly independent, since \linebreak $e_1e_2\odot_1e_1e_2 = -4e_3e_3+ \tfrac{8}{7}\id_{W}$ and $e_1e_2 \odot_2 e_1e_2 = 2e_1e_1+2e_2e_2-\tfrac{8}{7}\id_{W}$. They span the product space because it is $2$-dimensional (see \cref{reptheoryremark}).
\end{proof}
	
\subsection{Calculating parameters}

Before proving the main result of this section we have a technical lemma that will be useful both in this case, and in the $F_4$ case.

The following proof is due to an anonymous referee.
\begin{lemma}\label{stupidtechstuff} Let $k$ be an algebraically closed field. For $a,b\in W$ we have
	\[ \sum_{i=1}^{7} (e_i*a) (e_i*b) = 2\langle a, b\rangle\sum_{i=1}^{7}e_ie_i - 4ab \text{.}\]
\end{lemma}
\begin{proof}
	Note that $\sum e_ie_i = \sum e_i\langle e_i,\cdot\rangle = 2\id_W$. Then the left hand side of the equation in the lemma is equal to \[ \sum_{i=1}^{7} (e_i*a) (e_i*b) = \sum_{i=1}^{7} \tfrac{1}{2}((e_i*a )\langle e_i*b, \cdot \rangle + (e_i*b) \langle e_i*a, \cdot \rangle  ) = (-R^*_a\bullet R^*_b)\circ 2\id_W = -2R^*_a\bullet R^*_b, \] as the Malcev product is skew symmetric with respect to $\langle \cdot,\cdot \rangle$. By \cref{technicallem}~\ref{lem:technicallem:item2}, the last expression is precisely equal to the right hand side of the equation in the lemma.
\end{proof}
%
%
\begin{remark}\label{f4technicallemma}
	Using this, one can also prove that for any orthonormal basis $B$ of the octonions (with unit this time!), we have for $a,b\in \Oct$:
	\[ \sum_{x\in B} (x\cdot a)(x\cdot b) = \tfrac{1}{2}\langle a,b \rangle \sum_{x\in B} xx = \sum_{x\in B} (a\cdot x)(b\cdot x)\text{.} \]
\end{remark}
We have finally gathered enough information to prove the main theorem of this section.
\begin{theorem}\label{identificationag2}
	The algebra $A(\g_2)$ is isomorphic to the symmetric square of the pure octonions $\Sq^2W$, with multiplication given by
	\begin{align*}
		ab \star cd & = \tfrac{1}{12}\left( \langle a,c \rangle bd + \langle a,d \rangle bc +\langle b,c \rangle ad +\langle b,d \rangle ac + \langle a,b \rangle cd +\langle c,d \rangle ab\right) \\
		&\hspace{3ex}-\tfrac{1}{48} \left( {(a* c)}{(b* d)} + {(a* d)}{(b* c)} \right)\text{.}
	\end{align*}
	
\end{theorem}
\begin{proof}
	By \cref{symmoperatorssymsquare,symmopiso1}, it only remains to show that
	 $\star$ satisfies the formula above.
	
	We may assume without loss of generality that $k$ is algebraically closed. We denote a standard basis as in \cref{standardbasis}.
	
	We will determine explicitly the product $ab\star cd$ for $a,b,c,d$ pure octonions. By \cref{eq:multag2}, we have
	\begin{multline}\label{eq:multmaintheorem}
		ab\star cd = \left(\tfrac{1}{7^2}\langle a,b\rangle \langle c,d \rangle +f\left(ab-\tfrac{1}{7}\langle a,b\rangle\id_W,cd-\tfrac{1}{7}\langle c,d \rangle\id_W\right)\right)\id_W\\ +\tfrac{1}{7}\langle a,b\rangle \left(cd-\tfrac{1}{7}\langle c,d \rangle\id_W\right)+\tfrac{1}{7}\langle c,d \rangle \left(ab-\tfrac{1}{7}\langle a,b\rangle\id_W\right) \\+ \left(ab-\tfrac{1}{7}\langle a,b\rangle\id_W\right)\odot \left(cd-\tfrac{1}{7}\langle c,d \rangle\id_W\right)\text{.}
	\end{multline}
	The multiplication $\odot$ on $\sigma(V)$ should be a linear combination of the products $\odot_1$ and $\odot_2$, defined in \cref{prop:g2products}. By comparing $\odot_1$ and $\odot_2$ with \cref{computationg2}, we determine that $\odot = -\tfrac{1}{48}\odot_1 + \tfrac{1}{12}\odot_2$. Explicitly, we have for $ab,cd\in \sigma(V)$:
	\begin{multline}\label{eq:maintheoremodot}
		ab \odot cd  = \tfrac{1}{12}\left( \langle a,c \rangle bd + \langle a,d \rangle bc +\langle b,c \rangle ad +\langle b,d \rangle ac \right) \\
		-\tfrac{1}{48} \left( {(a* c)}{(b* d)} + {(a* d)}{(b* c)} \right) - \tfrac{5}{24\cdot 7}(\langle a,c \rangle\langle b,d \rangle +\langle a,d \rangle \langle b,c \rangle)\id_W\text{.}
	\end{multline}
	We can extend this product to the entire symmetric square by this same formula, and we will denote the extension by $\odot$ as well.
	\cref{stupidtechstuff} gives us
	\begin{align*}
	\id_{W} \odot ab &=\sum_{i}  -\tfrac{1}{48}( (e_i*a)(e_i*b))  -\tfrac{2}{48\cdot 7}(\langle e_i,a\rangle \langle e_i,b\rangle )\id_{W}  \\
		&\hspace{3ex}+\tfrac{1}{12}( \langle e_i,a\rangle e_ib+\langle e_i,b\rangle e_ia ) -\tfrac{2}{12\cdot7}(\langle e_i,a\rangle \langle e_i,b\rangle )\id_{W}\\
		&=  -\tfrac{1}{24}(2\langle a,b\rangle \id_{W}-2ab) -\tfrac{1}{12\cdot 7}\langle a,b\rangle \id_{W}   + \tfrac{1}{3}ab - \tfrac{1}{21}\langle a,b\rangle \id_{W}\\
		&= \tfrac{5}{12}ab - \tfrac{1}{ 7}\langle a,b \rangle \id_{W}\text{.}
	\end{align*} 
	For generic $ab,cd \in \Sq^2W$ we get (using \eqref{eq:maintheoremodot})
\begin{multline}\label{eq:maintheoremproduct}
		\left(ab - \tfrac{1}{7}\langle a,b\rangle\id_{W} \right)\odot \left(cd - \tfrac{1}{7}\langle c,d\rangle\id_{W} \right)\\
		 \begin{aligned}
			=&ab\odot cd + \tfrac{1}{7^2}\langle a,b\rangle \langle c,d \rangle\id_{W}\odot \id_{W}-\tfrac{1}{7}\langle c,d\rangle \id_{W}\odot ab -\tfrac{1}{7}\langle a,b\rangle \id_{W}\odot cd 
		 \end{aligned}\\
		 \begin{aligned}
			=&  ab\odot cd  -\tfrac{1}{12\cdot 7}\langle a,b\rangle \langle c,d \rangle\id_{W}\\&-\tfrac{1}{7}\langle c,d\rangle \left(\tfrac{5}{12}ab - \tfrac{1}{7}\langle a,b \rangle \id_{W}\right) -\tfrac{1}{7}\langle a,b\rangle \left(\tfrac{5}{12}cd - \tfrac{1}{7}\langle c,d \rangle \id_{W}\right)
		 \end{aligned}\\
		\begin{aligned}			
		=&\tfrac{17}{12\cdot 7^2}\langle a,b\rangle \langle c,d \rangle\id_{W} -  \tfrac{5}{ 12\cdot7}\langle a,b \rangle cd -\tfrac{5}{12\cdot7}\langle c,d \rangle ab\\
		&+ \tfrac{1}{12}\left( \langle a,c \rangle bd + \langle a,d \rangle bc +\langle b,c \rangle ad +\langle b,d \rangle ac \right) \\
		&-\tfrac{1}{48} \left( {(a* c)}{(b* d)} + {(a* d)}{(b* c)} \right) - \tfrac{5}{24\cdot 7}(\langle a,c \rangle\langle b,d \rangle +\langle a,d \rangle \langle b,c \rangle)\id_W\text{.}		
		\end{aligned} 
	\end{multline}
	Now, the only thing that remains is to determine the invariant bilinear form $f$ in \cref{eq:multmaintheorem}. However, we again know that there is only one invariant symmetric bilinear form (\cref{reptheoryremark}) on the irreducible $27$-dimensional representation of $G_2$, defined up to a scalar. We can thus see that
	\[f(ab,cd) = \lambda (\langle a,c \rangle\langle b,d \rangle +\langle a,d \rangle \langle b,c \rangle) \text{,}\]
	where $\lambda$ is a scalar factor. Again using \cref{computationg2} above, we get that $\lambda = \tfrac{5}{24\cdot 7}$. 
	We have
	\begin{align*}
		f(\id_{W},ab) &=\tfrac{5}{12\cdot 7} \langle a,b \rangle \text{,}
	\end{align*}
	thus
	\begin{multline}\label{eq:maintheorembilform}
		f\left(ab-\tfrac{1}{7}\langle a,b\rangle\id_W,cd-\tfrac{1}{7}\langle c,d \rangle\id_W\right)\\
		=\tfrac{5}{24\cdot 7} (\langle a,c \rangle\langle b,d \rangle +\langle a,d \rangle \langle b,c \rangle) -\tfrac{10}{12\cdot 7^2} \langle a,b \rangle \langle c,d\rangle + \tfrac{5}{12\cdot7^2} \langle a,b \rangle \langle c,d\rangle\\
		=\tfrac{5}{24\cdot 7}(\langle a,c \rangle\langle b,d \rangle +\langle a,d \rangle \langle b,c \rangle - 2 \langle a,b \rangle \langle c,d\rangle).
	\end{multline}
	Plugging \cref{eq:maintheorembilform,eq:maintheoremproduct} into \cref{eq:multmaintheorem} and cancelling out terms, we obtain the formula in the statement of the theorem.
\end{proof} 

\section{The automorphism group of type $G_2$}\label{secautg2}

In \cite[Remark~9.2]{chayet2020class}, it was noted we do not know the automorphism group yet of the algebra $A(\g_2)$. This section aims to resolve that issue. The techniques used in this section are inspired by \cite{GG15}.
In this section, we work over an algebraically closed field $k$, so we can conflate smooth algebraic groups with their $k$-points (\cite[Corollary~1.17 and Proposition~1.26]{milne2017algebraic}). We continue with the notation of the previous section, so $V$ denotes the $27$-dimensional irreducible representation, and $G$ is the (adjoint) group of type $G_2$ associated to $A(\g_2)$. The assumption on the characteristic remains the same as in the previous section, i.e.\@ $\kar k =0$ or $\kar k>7$.

We will need the following fact:
\begin{proposition}\label{innerautomorphism}
	Any automorphism of $G$ is inner.
\end{proposition}
\begin{proof}
	This follows from \cite[Theorem~27.4]{HumphreysLAG} and the fact that the Dynkin diagram of $G_2$ has no symmetries.
\end{proof}
We will denote the identity component of the stabilizer of the nondegenerate bilinear form $\tau|_V$ by $B = \SO(V)$.

The work done in \cite{Seitz} is instrumental in our argument. We lay out the results we need from \cite{Seitz} in the following proposition.
\begin{proposition}\label{lem:seitz}
	Let $G$ be a group of type $G_2$ and $V$ its unique irreducible $27$-dimensional representation. Suppose $G< H \leq\SL(V)$, and moreover that $H$ is smooth and connected. Then one of the following occurs:
	\begin{enumerate}
		\item $H= \SL(V)$,
		\item $H = \SO(V)$,
		\item $H$ is of type $B_3$ and acts on $V$ with highest weight $2\omega_1$,
		\item $H$ is of type $E_6$ and acts on $V$ with highest weight $\omega_6$.
	\end{enumerate}
\end{proposition}
\begin{proof}
	This is \cite[Theorem~2]{Seitz}, in case $X$ is of type $G_2$ and $V$ is the $27$-dimensional irreducible representation of highest weight $2\omega_1$.
\end{proof}
For the $B_3$ case we can prove uniqueness. We will first need an explicit model of the representation of highest weight $2\omega_1$. Note that any group of type $B_3$ acting faithfully on this representation has to be isomorphic to $\SO_7$, the adjoint group of type $B_3$.
\begin{lemma}\label{lem:repB3-2w1}
	Let $H$ be a group of type $B_3$. Let $W'$ be the $7$-dimensional representation of $H$. Denote the associated bilinear form (unique up to a scalar) by $\langle \cdot , \cdot \rangle'$. Then construct the symmetric square $\Sq^2W'$ and restrict to the subspace 
	\[V_H = \left\{ \sum_{i} a_ib_i \in \Sq^2W' \,\middle\vert\, \sum_{i} \langle a_i,b_i \rangle' =0 \right\}\text{.}\]
	This is the representation of $B_3$ with highest weight $2\omega_1$. It comes equipped with the $H$-equivariant bilinear form
	\[ \tau_H(ab,cd) = \tfrac{5}{7\cdot24}(\langle a,c\rangle' \langle b,d\rangle' + \langle a,d\rangle' \langle b,c\rangle') \text{.} \]
\end{lemma}
\begin{proof}
	The symmetric square of the natural representation of $B_3$ is equal to $V_H\oplus k$. Then we can compare the characters of the representations to check this representation is the irreducible representation of highest weight $2\omega_1$. 
\end{proof}
When we have two groups of type $G_2$ inside of a group isomorphic to $\SO(W)$, we can always prove they are conjugate. The reference for this fact was brought to our attention by Skip Garibaldi.
\begin{lemma}\label{lem:conjugateg2s}
	Suppose $G,G'$ are two (adjoint) groups of type $G_2$ contained in a group $\SO(W)$, with $W$ a $7$-dimensional vector space. Then $G$ and $G'$ are conjugate in $\SO(W)$.
\end{lemma}
\begin{proof}
	This is a special case of \cite[Theorem~1 p.14]{Malcevconjugacy}, originally proven by Frobenius. The proof is done over the complex numbers and for Lie groups instead of algebraic groups, though the proof still holds as long as $\kar k\neq 2$ and the field has square roots, and it also works for algebraic groups.
\end{proof}
\begin{proposition}\label{prop:onlyb3possible}
	The group $G$ is contained in exactly one group $B'$ isomorphic to $\SO_7$ in $B$.
\end{proposition}
\begin{proof}
	\underline{1. Existence:}
	
	As we realised $A(\g_2) = \Sq^2 W$ as the symmetric square of the pure octonions, we can consider the action of $B' = \SO(W)$ stabilising the bilinear form on the pure octonions. Then this stabilises $\tau|_V$ in particular.
	
	\underline{2. Uniqueness:}
	
	Suppose $H\cong \SO_7$ is another overgroup of $G$ in $B$. Then it also has to stabilise the bilinear form $\tau|_V$.
	
	Let $W',V_H,\tau_H$ be as in \cref{lem:repB3-2w1}. Then the representation $\SO(W')\to \GL(V_H)$ is isomorphic to the representation $H\to \GL(V)$ by \cref{lem:repB3-2w1,lem:seitz}. Thus we can find an isomorphism $\varphi:V_H\to V$ with $H = \varphi \SO(W') \varphi^{-1}$. Moreover, since both  $H$ and $\SO(W')$ stabilise only a $1$-dimensional space of bilinear forms, the isomorphism sends $\tau_H$ to a scalar multiple of $\tau$. We can assume without loss of generality that the isomorphism sends $\tau_H$ to $\tau$.
	
	Now pick an orthonormal basis $a_1,\dots,a_7 \in W'$ and an orthonormal basis $b_1,\dots,b_7 \in W$. Then we have an element $\rho \in \Orth(V)$ such that
	\begin{align*}
		\rho\left(\phi((a_1+a_i)(a_1-a_i))\right) &= (b_1+b_i)(b_1-b_i)\text{,} \\
		\rho\left(\phi(a_ia_j)\right) &= b_ib_j\text{.}
	\end{align*}
	The map $\rho$ is well defined by \cref{lem:constructionssymmsquare}\ref{item2:lem:constructionssymsquare}. But then we have $B' = \rho A \rho^{-1}$, and conjugation by $\rho$ stabilises $B\cong\SO(V)$, as $\SO(V)$ is normal in $\Orth(V)$.
	
	Then $G, \rho G\rho^{-1}$ are two groups of type $G_2$ contained in $\SO(W)$. By \cref{lem:conjugateg2s}, we can find a $\psi \in \GL(W)$ such that conjugation by $\psi$ stabilises $\SO(W)$ and sends $\rho G \rho^{-1}$ to $G$.
	
	The automorphism $\psi$ also acts on $V$, by sending $ab$ to $\psi(a)\psi(b)$. We will denote this map by $\psi$ as well.
	
	Composing $\phi,\rho$ and $\psi$, we find an element $\theta \in \GL(V)$ such that conjugation by $\theta$ sends $A$ to $B'$ and stabilises $G$.
	
	However, by \cref{innerautomorphism}, this means that there is an $h\in G$ such that for all $g\in G$
	\[ \theta g\theta^{-1} = hgh^{-1}\text{.} \] 
	
	By replacing $\theta$ with $h^{-1}\theta$, we can assume $\theta$ commutes with $G$. By Schur's Lemma then, $\theta$ acts as a scalar multiplication on $V$, since $G$ acts irreducibly on $V$. But then conjugation by $\theta$ is simply the identity isomorphism, so we conclude that $A=B'$.
\end{proof}
This proposition essentially gives us a very strong uniqueness property for case (iii) in \cref{lem:seitz}. Using this, we can prove the automorphism group of the algebra is precisely $G$.
\begin{proposition}\label{b3notcontained}
	The algebra product is not invariant under $\Lie(B')$. Then the group $B'$ does not stabilise the algebra product either.
\end{proposition}
\begin{proof}

As we have an explicit action of $B'$ on $V$, we can compute its Lie algebra by the Lie functor. We get
\[\Lie(B') = \left\{ \left(\sum_i a_ib_i\mapsto \sum_i (D'(a_i)b_i + a_iD'(b_i))\right) \in \End(V) \mid D' \in \mathfrak{so}(W) \right\}.\]
It is easy to compute the algebra product is not invariant under $\Lie(B')$. Indeed, take a standard basis as in \cref{standardbasis} (this is possible since $k$ is algebraically closed). Let $T\in \mathfrak{so}(W)$ be defined by sending $e_1$ to $e_2$, $e_2$ to $-e_1$ and all other basis vectors to $0$. Then $2T(e_1)e_1 \diamond e_4e_4 + 2e_1e_1 \diamond T(e_4)e_4 = -\tfrac{1}{3}e_5e_6 + 0\neq 0$. If the algebra product would then be invariant under $B'$, it should also be Lie invariant under $\Lie(B')$ but this is not the case. 
\end{proof}
The author wishes to thank Skip Garibaldi for clarifying the arguments made in \cite[Lemma~5.1]{GG15}, and providing substantial comments on the following proof.
\begin{corollary}\label{normalityg2}
	The group scheme $G$ is the identity component of $\Aut(A(\g_2))$. In particular, it is a normal subgroup, and $\Aut(A(\g_2))$ is smooth.
\end{corollary}
\begin{proof}
	To prove this, we specialise (and expand on) the argument of \cite[Lemma 5.1]{GG15} to this particular case.
	Let us look at the identity component $S$ of $\Aut(A(\g_2))$. Any automorphism has to stabilise the bilinear form $\tau$ by \cite[Example A.6]{chayet2020class}. In particular, we have the reduced subgroup $S_{\mathrm{red}} \leq S \leq \SO(V) \leq \SL(V)$. Since $G$ is connected and reduced, it is also contained in $S_{\mathrm{red}}$. Note that $S_{\mathrm{red}}$ is smooth by \cite[Proposition~1.26]{milne2017algebraic}. If $S_{\mathrm{red}}$ is not equal to $G$, then by \cref{lem:seitz} it is either of type $B_3$ with highest weight $2\omega_1$, of type $E_6$ with highest weight $\omega_6$ or equal to $\SO(V),\SL(V)$. But none of these possibilities can occur, since $\SO(V),\SL(V)$ and type $E_6$ do not stabilise algebra products on this representation, and type $B_3$ cannot occur by \cref{prop:onlyb3possible,b3notcontained}.
	
	We still need to show that $\Lie(S)=\Lie(X)$ to prove the group is smooth. The argument in the second half of \cite[Lemma~5.1]{GG15} ensures that if $\Lie(S)\neq\Lie(X)$, we can find a simple simply connected algebraic group $H$ and a representation $\phi\colon H \to \SL(V)$ such that $\diff \phi(\Lie(H))$ is contained in $\Lie(S)$, and $\phi(H(k))$ has to be one of the options in \cref{lem:seitz}, so $H$ has to be the corresponding simple simply connected algebraic group.
	
	Now note that the algebra product $\odot$ on $V$ and $\tau|_V$ are Lie invariant under $\diff \phi(\Lie(H)) \leq \Lie(S)$.  For $H=\Spin(V),\SL(V)$ or $H$ of type $E_6$, this is impossible by \cite[Lemma~2.4]{GG15} (see also \cite[Table~A]{1diminvariants}). Type $B_3$ is impossible by the following reasoning: suppose that $H$ is of type $B_3$. As $G\leq \phi(H)$, and $G$ and $H$ both only stabilise a $1$-dimensional space of bilinear forms on $V$, we need $\phi(H) \leq \SO(V,\tau)$. But this means $\phi(H(k)) = B'$ by \cref{prop:onlyb3possible}, so $\diff \phi(\Lie(H)) = \Lie(B')$. But by \cref{b3notcontained}, we obtain a contradiction. Thus none of the possibilities for $H$ are possible, and $\Lie(S) = \Lie(X)$. \end{proof}
\begin{theorem}\label{autG2} We have
	\[\Aut(A(\g_2)) = G\text{.} \]
\end{theorem}
\begin{proof}
	Let $\theta \in \Aut(A(\g_2))\setminus G$. Then, because of \cref{normalityg2}, conjugation by $\theta$ is an automorphism of $G_2$. By \cref{innerautomorphism}, this means that there is an $h\in G$ such that for all $g\in G$
	\[ \theta g\theta^{-1} = hgh^{-1} \] 
	holds. In other words, $h^{-1}\theta$ commutes with $G$. 
	
	But if $\phi =h^{-1}\theta$ commutes with $G$, this also holds after resticting to $V = V(2\omega_1)$.  This means $\phi$ acts as a scalar $a$ on $V$ by Schur's Lemma.
The multiplication on $V$ is non-zero, and $\phi$ stabilises this multiplication, so we have $a^2=a$.
$\phi$ is in particular a bijection,thus $a=1$ and $\phi$ is the identity. This implies $\theta = h$, which contradicts our choice of $\theta$.
\end{proof}

\section{The algebra of type $F_4$}\label{secf4}
We can reuse our recipe for the treatment of the $G_2$ case to obtain similar results for $F_4$. However, when trying to define multiplications on the irreducible representation of highest weight $2\omega_4$, some more care is needed, as hinted towards in the introduction.

Recall that the assumpiton on the characteristic in this section is $\kar k =0$ or $\kar k >13$. From this point onwards, $W$ is the $26$-dimensional natural representation of $F_4$, and $V$ is the $324$-dimensional irreducible representation of highest weight $2\omega_4$ of $F_4$.

For $F_4$, the embedding from \cref{embedding} becomes, using \cref{casimirisscalar},
\begin{equation}\label{multiplicationformula}
	\sigma(S(XY)) = 54X\bullet Y - \tfrac{3}{2}\Tr(XY)\id_{W} \text{.}
\end{equation}
As in the $G_2$ case, we do not write the embedding $\pi$ of the Lie algebra explicitly.

This part is dedicated to studying the embedding $\sigma$ more deeply, using the Albert algebras. We will not be able to find an explicit formula for $\sigma$ in terms of the standard derivations and the bilinear form on the Albert algebras, as in the $G_2$ case. But we will derive enough formulas to compute a proposition similar to \cref{computationg2} for $G_2$. In this way, we find an alternate description for the algebra $A(\mathfrak{f}_4)$. Our main result is \cref{algF4}. In this section, we will denote the \emph{split} octonions by $\Oct$.

\subsection{Computing some derivations and their Jordan products}
In this section, we will compute all the derivations and their Jordan products necessary in \cref{parametersection}.
\begin{lemma}
	Let $k$ be algebraically closed, and $A=\mathcal{H}_3(\mathbb{O})$ an Albert algebra.
	\begin{enumerate}
		\item For $i,j \in \{1,2,3\}$, we have $D_{\mathbf{1}_i,\mathbf{1}_j}=0$,
		\item for $i \in \{1,2,3\}$ and $a \in \mathbb{O}$, we have $D_{\mathbf{1}_i,a_i}=0$,
		\item for $i\neq j \in \{1,2,3\}$ and $a \in \mathbb{O}$, we have
		\begin{enumerate}[label = {\rm (\alph*)}]
			\item $D_{\mathbf{1}_i,a_j}(\mathbf{1}_i) = -\tfrac{1}{4}a_j$,
			\item $D_{\mathbf{1}_i,a_j}(\mathbf{1}_j) = 0$,
			\item $D_{\mathbf{1}_i,a_j}(\mathbf{1}_k) = \tfrac{1}{4}a_j$ with $k\neq i,j$,
			\item $D_{\mathbf{1}_i,a_j}(b_i) = \tfrac{1}{2}a_j\cdot b_i$ with $b\in \Oct$,
			\item $D_{\mathbf{1}_i,a_j}(b_j) = \tfrac{1}{4}\langle a,b\rangle (\mathbf{1}_i-\mathbf{1}_k)$ with $b\in \Oct$,
			\item $D_{\mathbf{1}_i,a_j}(b_k) = -\tfrac{1}{2}a_j\cdot b_k$ with $b\in \Oct$ and $k\neq i,j$,
		\end{enumerate}
		\item for $ i\in \{1,2,3\}$ and $a,b\in \mathbb{O}$, we have
		\begin{enumerate}[label = {\rm (\alph*)}]
			\item $D_{a_i,b_i}(\mathbf{1}_i) = 0$,
			\item $D_{a_i,b_i}(\mathbf{1}_j) = 0$ with $j\neq i$,
			\item $D_{a_i,b_i}(c_i) = \tfrac{1}{2}(\langle b,c\rangle a_i - \langle a,c\rangle b_i)$ with $c\in \Oct$,
			\item $D_{a_i,b_i}(c_j) = a_i\cdot(b_i\cdot c_j)-b_i\cdot (a_i\cdot c_j)$ with $b\in \Oct$ and $j\neq i$.
		\end{enumerate}
	\end{enumerate}
\end{lemma}
\begin{proof}
	This can be computed using \cref{albertmultiplication}.
\end{proof}
\begin{lemma}
Let $k$ be algebraically closed, and $A=\mathcal{H}_3(\mathbb{O})$ an Albert algebra. Suppose $\{i,j,k\} = \{1,2,3\}$ and $a,b,c\in \mathbb{O}$. Then we have
	\begin{enumerate}
		\item
		\begin{enumerate}[label = {\rm (\alph*)}]
			\item 
			
			$\displaystyle D_{\mathbf{1}_i,a_j}\bullet D_{\mathbf{1}_i,b_j}(\mathbf{1}_i-\mathbf{1}_j)= -\tfrac{1}{16}\langle a,b\rangle(\mathbf{1}_i-\mathbf{1}_k)$,
			
			\item 
			
			$\displaystyle D_{\mathbf{1}_i,a_j}\bullet D_{\mathbf{1}_i,b_j}(\mathbf{1}_i-\mathbf{1}_k)= -\tfrac{1}{8}\langle a,b\rangle(\mathbf{1}_i-\mathbf{1}_k)$,
			
			\item 
			
			$\displaystyle D_{\mathbf{1}_i,a_j}\bullet D_{\mathbf{1}_i,b_j}(c_i)= -\tfrac{1}{32}\langle a,b\rangle c_i$,
			
			\item
			
			$\displaystyle D_{\mathbf{1}_i,a_j}\bullet D_{\mathbf{1}_i,b_j}(c_j)= -\tfrac{1}{16}(\langle a,c\rangle b_j +\langle b,c\rangle a_j)$,
			
			\item
			
			$\displaystyle D_{\mathbf{1}_i,a_j}\bullet D_{\mathbf{1}_i,b_j}(c_k)= -\tfrac{1}{32}\langle a,a\rangle c_k$,
		\end{enumerate}
		\item 
		\begin{enumerate}[label = {\rm (\alph*)}]
			\item
			
			$\displaystyle D_{\mathbf{1}_i,a_j}\bullet D_{\mathbf{1}_k,b_j}(\mathbf{1}_i-\mathbf{1}_j)= \tfrac{1}{16}\langle a,b\rangle(\mathbf{1}_i-\mathbf{1}_k)$,
			
			\item
			
			$\displaystyle D_{\mathbf{1}_i,a_j}\bullet D_{\mathbf{1}_k,b_j}(\mathbf{1}_i-\mathbf{1}_k)= \tfrac{1}{8}\langle a,b\rangle(\mathbf{1}_i-\mathbf{1}_k)$,
			
			\item 
			
			$\displaystyle D_{\mathbf{1}_i,a_j}\bullet D_{\mathbf{1}_k,b_j}(c_i)= \tfrac{1}{32}\langle a,b\rangle c_i$,
			
			\item 
			
			$\displaystyle D_{\mathbf{1}_i,a_j}\bullet D_{\mathbf{1}_k,b_j}(b_j)= \tfrac{1}{16}(\langle a,c\rangle b_j +\langle b,c\rangle a_j) $,
			
			\item 
			
			$\displaystyle D_{\mathbf{1}_i,a_j}\bullet D_{\mathbf{1}_k,b_j}(c_k)= \tfrac{1}{32}\langle a,b\rangle c_k$,
		\end{enumerate}
		\item 
		\begin{enumerate}[label = {\rm (\alph*)}]
			\item
			
			$\displaystyle D_{a_i,b_i}^2(\mathbf{1}_i-\mathbf{1}_j)= 0$,
			
			\item 
			$\displaystyle
				D_{a_i,b_i}^2(c_i) = \tfrac{1}{4}\langle b,c \rangle \langle b,a \rangle a_i 
				-	\tfrac{1}{4}\langle b,c \rangle \langle a,a \rangle b_i
				- \tfrac{1}{4}\langle a,c \rangle \langle b,b \rangle a_i 
				+ \tfrac{1}{4}\langle a,c \rangle \langle b,a \rangle b_i$,

			\item 
			$\displaystyle D_{a_i,b_i}^2(c_j)= \tfrac{1}{16}\left(\langle a,b\rangle^2 - \langle a,a\rangle\langle b,b\rangle \right)c_j$.
		\end{enumerate}		
	\end{enumerate}
\end{lemma}
\begin{proof}
	There is only one case that is hard to compute, and it is $D_{a_i,b_i}(c_j)$, with $i\neq j$ and $a,b,c\in \Oct$. We do this one as an example, the others are analogous but easier. We also assume $i,j,k$ is a cyclic permutation of $1,2,3$; the other case is completely analogous, but we would have to multiply the octonions to the right instead of to the left in the upcoming computation. This does not make a difference for the end result.
	\begin{align*}
		D_{a_i,b_i}^2(c_j) &= D_{a_i,b_i}\left( \tfrac{1}{4} (\overline{a}\cdot(b\cdot c)- \overline{b}\cdot(a\cdot c))_j \right) \\
		&= \tfrac{1}{16}\left( \left(\overline{a}\cdot(b\cdot(\overline{a}\cdot(b\cdot c)- \overline{b}\cdot(a\cdot c))) \right)_j\right.\\
		&\hspace{3ex}\left. -\left(\overline{b}\cdot(a\cdot(\overline{a}\cdot(b\cdot c)- \overline{b}\cdot(a\cdot c)))\right)_j \right) \\
		&= \tfrac{1}{16}\left( -2N(a)N(b) + \left(\overline{a}\cdot(b\cdot(\overline{a}\cdot(b\cdot c)))+ \overline{b}\cdot(a\cdot(\overline{b}\cdot(a\cdot c)))\right)_j \right)
	\end{align*}
	We compute these last terms using the Moufang identities (\cref{octonionidentities}\ref{moufangidentities}).
	\begin{align*}
		(\overline{a}\cdot (b \cdot (\overline{a} \cdot ( b\cdot c )))) &= ((\overline{a}\cdot b)\cdot \overline{a})\cdot(b\cdot c) = ((\langle a,b\rangle e - \overline{b}\cdot a)\cdot \overline{a})\cdot(b\cdot c)\\
		&=\langle a,b\rangle \overline{a}\cdot(b\cdot c) - N(a)N(b) c\text{.}
	\end{align*}
	By switching the roles of $a$ and $b$ we get (using \cite[Lemma~1.3.3(iii)]{SpringerVeldkamp} in the last step)
	\begin{align*}
		\overline{a}\cdot(b\cdot(\overline{a}\cdot(b\cdot c)))+ \overline{b}\cdot(a\cdot(\overline{b}\cdot(a\cdot c))) &= \langle a,b\rangle ( \overline{a}\cdot(b\cdot c) + \overline{b}\cdot(a\cdot c)  ) - 2 N(a)N(b)c\\
		&=  \langle a ,b \rangle^2 c -  2N(a)N(b)c\text{.}
	\end{align*}
 	This finally results in 
	\begin{equation*}
		D_{a_i,b_i}^2(c_j)= \tfrac{1}{16}\left(\langle a,b\rangle^2 - \langle a,a\rangle\langle b,b\rangle \right)c_j\text{.}\qedhere
	\end{equation*}	
\end{proof}
We can bundle this lemma into some compact formulas. Similarly to the $G_2$ case, we use the notation from \cref{rem:shorthand}.

\begin{definition} Let $k$ be algebraically closed.
	Let $B$ be an orthonormal basis for the (split) octonions $\mathbb{O}$. We define for $i\in \{1,2,3\}$ the operator $I_i$ as
	\[ I_i = \sum_{x\in B} x_ix_i \text{.}  \]
	Note that the operator $I_i$ acts as an indicator function for the subspace $\{ a_i \vert a\in \mathbb{O} \}$ in the Albert algebra.
\end{definition}

\begin{corollary}\label{embeddingformulasf4} Let $\{i,j,k\} = \{1,2,3\}$ and $a,b\in \mathbb{O}$. We have the following identities:
	\begin{enumerate}\itemsep0.5em
		\item $\begin{aligned}[t]
		\sigma(S(D_{a_i,b_i}^2)) &= \tfrac{9}{4}\Big( \tfrac{3}{2}(I_j+I_k)(\langle a,b\rangle^2 - \langle a,a \rangle \langle b,b \rangle) +12\langle a,b\rangle a_ib_i &\\
			&\hspace{5ex}- 6 \langle b,b\rangle a_i a_i -6 \langle a,a \rangle b_i b_i  - (\langle a,b \rangle^2 - \langle a,a\rangle \langle b,b\rangle)\id_W \Big) \text{,}&
		\end{aligned}$
		
		\item\label{embforf4:item2} $\begin{aligned}[t]
			\sigma(S(D_{\mathbf{1}_i,a_j}D_{\mathbf{1}_i,b_j})) &= -\tfrac{9}{8}\Big(\tfrac{3}{2}(I_j+I_k) \langle a,b \rangle +3\langle a,b\rangle (\mathbf{1}_i-\mathbf{1}_k)(\mathbf{1}_i-\mathbf{1}_k) & \\
				&\hspace{7ex}+6 a_jb_j - \langle a,b\rangle \id_W \Big)\text{,}&
		\end{aligned}$
		
		\item\label{embforf4:item3} $\begin{aligned}[t]
				\sigma(S(D_{\mathbf{1}_i,a_j} D_{\mathbf{1}_k,b_j}))&=\tfrac{9}{8}\Big(\tfrac{3}{2}(I_j+I_k) \langle a,b \rangle +3\langle a,b\rangle (\mathbf{1}_i-\mathbf{1}_k)(\mathbf{1}_i-\mathbf{1}_k) & \\
				&\hspace{5ex}+6 a_jb_j - \langle a,b\rangle \id_W \Big)\text{.}&
		\end{aligned}$ 
	\end{enumerate}
\end{corollary}
To find the parameters of the multiplication in \cref{parametersection}, we will explicitly compute the multiplication $\star: \sigma(A(\mathfrak{f}_4))\times \sigma(A(\mathfrak{f}_4))\to A(\mathfrak{f}_4)$ for certain specific elements in $A(\mathfrak{f}_4)$ with the formulas derived above.
\begin{proposition}\label{prop:excompf4}
Let $k$ be algebraically closed. Let $a,b\in \mathbb{O}$ be two isotropic octonions with $\langle a,b\rangle \neq 0$ and $\{i,j,k\}= \{1,2,3\}$. Then
\begin{equation}\label{ex:productf4}
	a_ja_j\star b_jb_j = \tfrac{1}{36}(12\langle a,b\rangle a_jb_j + \tfrac{3}{2}\langle a,b \rangle^2 (I_i+I_k) - \langle a,b\rangle^2 \id_W)\text{.}
\end{equation}
\end{proposition}
\begin{proof}
	In this case, we have
\begin{align*}
	\sigma(S(D_{\mathbf{1}_i,a_j}D_{\mathbf{1}_i,a_j})) 
	&= -\tfrac{27}{4} a_ja_j\text{.}
\end{align*}
by \cref{embeddingformulasf4}, and analogously, $\sigma(S(D_{\mathbf{1}_i,b_j}^2)) = -\tfrac{27}{4}b_jb_j$.

Now, by \cref{diamondprod} and \cref{casimirisscalar}
we have 
\begin{align*}
	S(X^2)\diamond S(Y^2) &= 9\left( S((\ad X)^2Y, Y ) + S((\ad Y)^2X,X) \right) \\
	&+9S\left([X,Y][X,Y]\right)+3\Tr(X,Y)S\left(XY\right)\text{.}
\end{align*}
To compute this with $X= D_{\mathbf{1}_i,a_j}$ and $Y=D_{\mathbf{1}_i,b_j}$, we first compute $[X,Y] = [D_{\mathbf{1}_i,a_j},D_{\mathbf{1}_i,b_j}]$.
\begin{align*}
	[D_{\mathbf{1}_i,a_j},D_{\mathbf{1}_i,b_j}] &= D_{D_{\mathbf{1}_i,a_j}(\mathbf{1}_i),b_j} + D_{\mathbf{1}_i,D_{\mathbf{1}_i,a_j}(b_j)} = -\tfrac{1}{4}D_{a_j,b_j} + \tfrac{\langle a,b \rangle}{4}D_{\mathbf{1}_i, \mathbf{1}_i-\mathbf{1}_k } \\
	&= -\tfrac{1}{4}D_{a_j,b_j}\text{.}
\end{align*}
Using this computation, we also get
\begin{align*}
	(\ad X)^2 Y &= -\tfrac{1}{4}[D_{\mathbf{1}_i,a_j},D_{a_j,b_j}] = -\tfrac{1}{4}\left( D_{D_{\mathbf{1}_i,a_j}(a_j),b_j} +D_{a_j,D_{\mathbf{1}_i,a_j}(b_j)}\right) \\
	&= -\tfrac{1}{4}\left( \tfrac{1}{4}\langle a,a\rangle D_{\mathbf{1}_i-\mathbf{1}_k,b_j}+ \tfrac{1}{4}\langle a,b \rangle D_{a_j,\mathbf{1}_i-\mathbf{1}_k} \right)\\
	&= \tfrac{1}{16}\langle a,b\rangle D_{\mathbf{1}_i-\mathbf{1}_k,a_j}\text{,}
\end{align*}
and reversing the roles of $X$ and $Y$, we have $(\ad Y)^2 X = \tfrac{1}{16}\langle a,b\rangle D_{\mathbf{1}_i-\mathbf{1}_k,b_j}$.

Plugging this information into \cref{multiplicationformula} we get
\begin{align*}
	S(X^2)\diamond S(Y^2) &= \tfrac{9}{16}\langle a,b\rangle \left( S(D_{\mathbf{1}_i-\mathbf{1}_k,a_j}, D_{\mathbf{1}_i,b_j} ) + S( D_{\mathbf{1}_i-\mathbf{1}_k,b_j},D_{\mathbf{1}_i,a_j}) \right) \\
	&\hspace{3ex}+\tfrac{9}{16}S\left(D_{a_j,b_j}^2\right)-\tfrac{9}{4}\langle a,b\rangle S\left(D_{\mathbf{1}_i,a_j},D_{\mathbf{1}_i,b_j}\right) \\
	&= \tfrac{9}{16}S\left(D_{a_j,b_j}^2\right)\text{,}
\end{align*}
where the last equality holds due to \cref{embeddingformulasf4}.
Now we let $\sigma$ act on both sides to get (using \cref{embeddingformulasf4})
\begin{align*}
	\sigma( S(X^2)\diamond S(Y^2) )&=\tfrac{9}{16}\cdot \tfrac{9}{4}\left(  12\langle a,b\rangle a_jb_j + \tfrac{3}{2}\langle a,b \rangle^2 (I_i+I_k) - \langle a,b\rangle^2 \id_W \right) \\
	&= \tfrac{27^2}{4^2} a_ja_j\star b_jb_j\text{.}
\end{align*}

So eventually we have \eqref{ex:productf4}.
\end{proof} 

\subsection{An embedding into the symmetric square of the Albert algebra}
Recall that $A(\mathfrak{f}_4)=k\oplus V$, where $V$ is the irreducible representation with highest weight $2\omega_4$.

By \cref{symmsquaresec}, we can embed the algebra into the symmetric square of the $26$-dimensional traceless Albert algebra
\[ \sigma \colon A(\mathfrak{f}_4)\hookrightarrow \Sq^2W\text{.} \]

However, when looking at the character of $\Sq^2W$, we see that 
\[ \Sq^2W \cong k \oplus V\oplus W \cong  A(\mathfrak{f}_4) \oplus W\text{.} \]
This is because the Albert product (after projecting to $W$) is non-zero. In fact, by this reasoning, $k\oplus V$ is the kernel of the Albert product. Thus we have

\[ \sigma(A(\mathfrak{f}_4)) = \left\{ \sum_i X_iY_i \in \Sq^2W \,\middle\vert\, \pi\left(\sum_i X_i\cdot Y_i\right) = 0 \right\}\text{,} \]

where $\pi\colon A\to W$ is the projection by the trace form, i.e. $\pi(a) = a- \tfrac{\langle a,e\rangle}{3}e$.

\begin{notation}
	We denote the projection after multiplication by $*$, i.e. for $X,Y\in A$ we have
	\[ X*Y \coloneqq \pi(X\cdot Y)\text{.} \]
\end{notation}

\subsection{Defining multiplications on $V$}
One can compute the characters of the representation $V$ (and of $\Sq^2V$) to deduce that the product space on $V$ is $2$-dimensional and that there is just one symmetric bilinear form (up to a scalar multiple) on $V$. By the previous section, we can actually deduce that
\[ \sigma(V) = \left\{ \sum_i X_iY_i \in \Sq^2W\,\middle\vert\, \sum_i X_i\cdot Y_i = 0 \right\}\text{.} \]
In the $G_2$ case, we had two fairly easy multiplications on the entire symmetric square. However, in this case we do not work with the entire symmetric space. We will need to be more careful in defining the multiplications, as products similar to the $G_2$ case would not be well-defined. 
\begin{remark}
	Suppose $\odot$ is a multiplication on $\sigma(A(\mathfrak{f}_4)$. Then for any $a,b,c,d\in W$ such that $ab,cd \in \sigma(A(\mathfrak{f}_4))$, we can write
\[ ab \odot cd  = \sum_i u_i(a,b,c,d)v_i(a,b,c,d), \]
with $u_i(a,b,c,d),v_i(a,b,c,d) \in  W$. We know that we will have $ \sum_i u_i(a,b,c,d)*v_i(a,b,c,d) =0$. This means we can try to find degree $4$ identities that are invariant under the permutations $(a,b)$, $(a,c)(b,d)$ and the group $G$, and hope to lift them to multiplications on $\sigma(A(\mathfrak{f_4}))$.
	This is how we found the multiplications in \ref{subsec:odot1} and \ref{subsec:odot2}.
\end{remark}

\subsubsection{Finding multiplication $\odot_1$.}\label{subsec:odot1}

To find the first multiplication, we translate \cref{equation512} to the projected product.
\begin{lemma}\label{lem:springerveldkamprojected}
	Let $x,y \in W$ be traceless Albert elements. Then
	\[(z*x)*y + (z*y)*x + (x*y)*z =  \tfrac{1}{6}( \langle z,x \rangle y +\langle z,y \rangle x +\langle x,y \rangle z  ) \text{.} \]
\end{lemma}
\begin{proof}
	By \cref{equation512}, we have
	\begin{align}
		(z\cdot x)\cdot y + (z\cdot y)\cdot x + (x\cdot y)\cdot z &=  \tfrac{1}{2}\langle x,y \rangle z + \tfrac{1}{2}\langle z,x \rangle y +\tfrac{1}{2}\langle z,y\rangle x 
		 + 3\langle z,x,y\rangle e \text{.} \label{eq3}
	\end{align}
	On the other hand we also have
	\begin{multline} \label{eq4}
		(z\cdot x)\cdot y + (z\cdot y)\cdot x + (x\cdot y)\cdot z \\
	\begin{aligned}
		 &= (z*x + \tfrac{1}{3}\langle z,x \rangle e)\cdot y + (z*y +\tfrac{1}{3}\langle z,y \rangle e ) \cdot x+(x*y +\tfrac{1}{3}e ) \cdot z \\
		&= (z*x)\cdot y + (z*y)\cdot x +(x*y)\cdot z + \tfrac{1}{3}\langle z,x \rangle y + \tfrac{1}{3}\langle z,y \rangle x + \tfrac{1}{3}\langle x,y \rangle z \text{.}
	\end{aligned}
	\end{multline}
	Equating \cref{eq3,eq4}, and then projecting to $W$ leaves us with the desired identity.
\end{proof}

\begin{lemma}\label{lem:5.6}
	For $a,b,c,d \in W$ traceless Albert elements, we have the following identity:
	\begin{multline*}
		\left( a*c\right)*\left( b*d\right) - \tfrac{1}{6}\Big( \langle b,d \rangle\left( a*c\right) +\langle c,a*b \rangle d +\langle a,c*d \rangle b\Big)
		+  b*\left(\left( a*c\right)*d\right) +  d*\left(\left( a*c\right)*b\right)
		 \\=0\text{.}
	\end{multline*}
\end{lemma}
\begin{proof}
	Using \cref{lem:springerveldkamprojected} with $x=a*c,y=b,z=d$ gives
	\begin{equation*}
		(a*c)*(b* d) + b*(d* (a*c)) + d*(b*(a*c)) = \tfrac{1}{6}\langle b,d \rangle a*c+ \tfrac{1}{6}\langle c,a*b \rangle d + \tfrac{1}{6}\langle a,c*d \rangle b\text{.}
	\end{equation*}
	Rearranging the terms immediately gives the desired identity.
\end{proof}

\begin{proposition}\label{prop:identitymult1} For $a_i,b_i,c_j,d_j \in W$ traceless elements with $i=1,\dots,n$ and $j=1,\dots,m$ such that $\sum_i a_i*b_i= \sum_j c_j*d_j=0$, we have the following identity:
	\begin{multline*} \sum_{i,j}\Big(2\left((a_i*c_j)*(b_i*d_j) + (a_i*d_j)*(b_i*c_j)\right)\\[-2ex] + \tfrac{1}{6}( \langle a_i,c_j\rangle b_i*d_j + \langle a_i,d_j\rangle b_i*c_j+\langle b_i,c_j \rangle a_i*d_j +\langle b_i,d_j\rangle a_i*c_j  )\Big) =0\text{.}\end{multline*}
\end{proposition}
\begin{proof}
	For simplicity of notation, we assume $m= n=1$. Plugging in $a,b,c,d$ into \cref{lem:5.6}, we get
	\begin{equation}\label{eq:prop:identitymult1}
		\left( a*c\right)*\left( b*d\right) - \tfrac{1}{6} \langle b,d \rangle\left( a*c\right)
		+  b*\left(\left( a*c\right)*d\right) +  d*\left(\left( a*c\right)*b\right) =0\text{.}
	\end{equation}
	By switching the roles of $a$ and $b$ in \cref{eq:prop:identitymult1}, switching the roles of $a,b$ and $c,d$, and doing both, we get $4$ different equations. Adding them all up we get
	\begin{align*}
		&  2\big((a*c)*(b*d) + (a*d)*(b*c)\big) \\
		&-\tfrac{1}{6}\left( \langle a,c\rangle b*d + \langle a,d\rangle b*c+\langle b,c \rangle a*d +\langle b,d\rangle a*c \right) \\
		&+a*((b*d)*c) +d*((a*c)*b) + b*((a*c)*d)+ c*((b*d)*a)\\
		&+a*((b*c)*d) + d*((b*c)*a) + b*((a*d)*c) + c*((a*d)*b) \\
		&=0\text{.}
	\end{align*}
	We can simplify the terms on the last two lines using \cref{lem:springerveldkamprojected} four times to get the desired identity.
\end{proof}
\begin{proposition}\label{prop:f4odot1}
	Define for $ab,cd\in \Sq^2W$ the product
	\begin{multline}\label{diamond1}
	ab\odot_1cd \coloneqq \tfrac{1}{6}\left( \langle a,c\rangle bd + \langle a,d\rangle bc+\langle b,c \rangle ad +\langle b,d\rangle ac  \right)
 \\+2\left((a*c)(b*d) + (a*d)(b*c)\right)
	-\tfrac{2}{3\cdot26}( \langle a,c\rangle\langle b,d\rangle + \langle a,d\rangle \langle b,c \rangle)\id_W\text{.}
	\end{multline}
	Then $\odot_1$ restricted to $\sigma(V)$ is a well-defined, $F_4$-equivariant commutative product on $\sigma(V)$.
\end{proposition}
\begin{proof}
	First we prove that
	\begin{multline*}
	\left(\sum_{i}a_ib_i\right)\odot_1|_{\sigma(V)} \left(\sum_{j}c_jd_j\right) \\\coloneqq \sum_{i,j}\Pi\Big(2\left((a_i*c_j)(b_i*d_j) + (a-i*d_j)(b_i*c_j)\right) \\
	+\tfrac{1}{6}( \langle a_i,c_j\rangle b_id_j + \langle a_i,d_j\rangle b_ic_j+\langle b_i,c_j \rangle a_id_j +\langle b_i,d_j\rangle a_ic_j  ) \Big) \text{,}
	\end{multline*}
	where $\Pi$ is the projection $\sigma(V)\oplus k \to \sigma(V)$, is an $F_4$-equivariant product with image in $\sigma(V)$.
	The fact that the image is in $\sigma(V)$ follows immediately from \cref{prop:identitymult1}.  The $F_4$-equivariance follows from the $F_4$-equivariance of both $*$, $\langle \cdot ,\cdot \rangle$ and $\Pi$.
	
	Next we prove that this definition is equal to \cref{diamond1} when restricted to $\sigma(V)$. We can write for $\sum_i a_ib_i,\sum_j c_jd_j\in \sigma(V)$
\begin{multline*}\label{diamond1}
	\left(\sum_{i}a_ib_i\right)\odot_1 \left(\sum_{j}c_jd_j\right)\\ \coloneqq  
	 \sum_{i,j}\Big(  \tfrac{1}{6}\left( \langle a_i,c_j\rangle b_id_j + \langle a_i,d_j\rangle b_ic_j+\langle b_i,c_j \rangle a_id_j +\langle b_i,d_j\rangle a_ic_j  \right)
 \\+2\left((a_i*c_j)(b_i*d_j) + (a_i*d_j)(b_i*c_j)\right)
	-f_1(a_ib_i,c_jd_j)\id_W\Big)\text{,}
	\end{multline*}	
where $f_1$ is a symmetric bilinear $F_4$-equivariant form on $\sigma(V)$. But we only have one such form (up to a scalar multiple), so the equality
\[ f_1(ab,cd)= \lambda_1( \langle a,c\rangle\langle b,d\rangle + \langle a,d\rangle \langle b,c \rangle) \]
holds for a certain $\lambda_1 \in k$. By computing this product for $a_1a_1$ and $b_1b_1$ where $a$ and $b$ are arbitrary isotropic octonions, we obtain
$\lambda_1 = \tfrac{2}{3\cdot26}$.
\end{proof}

\subsubsection{Finding multiplication $\odot_2$}\label{subsec:odot2}

To find the second multiplication, we are going to have to be a little more creative. First note that for any orthonormal basis $X$ of $W$, we have $\id_W = \sum_{x\in X} xx$ (this is \cref{lem:constructionssymmsquare}\ref{item1:lem:constructionssymsquare}). On the entire symmetric square, we can define the following product
\[ ab*cd \coloneqq (a*c)(b*d) + (a*d)(b*c)\text{.} \]
This product is clearly $F_4$-equivariant and symmetric. We will somehow try to modify this to something that is well-defined on the smaller subspace $V$.
\begin{proposition}\label{prop:identitymult2}
	Let $k$ be algebraically closed. Suppose $B$ is an orthonormal basis for the traceless Albert algebra $W$, and $a,c \in W$. Then we have the following identity:
	\[ \sum_{x\in B} (a*x)*(c*x) +a*c = 0\text{.}  \]
\end{proposition}
\begin{proof}
	First note that $\sum_{x\in B} xx = \id_V \in \sigma(A(\mathfrak{f}_4))$ by \cref{embedding}. Because of the considerations in the beginning of this section, this implies $\sum_{x\in B} x*x = 0$. We can also assume $a=c$, then the more general statement follows from linearisation.
	By \cref{lem:5.6}, we have
	\begin{multline}\label{eq:lem:identity2}
		\left( a*x\right)*\left( a*x\right) - \tfrac{1}{6}\Big( \langle a,x \rangle\left( a*x\right) +\langle x,a*a \rangle x +\langle a,x*x \rangle a\Big)
		\\+  a*\left(\left( a*x\right)*x\right) +  x*\left(\left( a*x\right)*a\right)
		 =0\text{.}
	\end{multline}
	By \cref{lem:springerveldkamprojected}, we have
	\begin{equation}\label{doublex}
		\sum_{x\in B} (a*x)* x = \tfrac{7}{3}a \text{,}
	\end{equation}
	and
	\begin{equation}\label{doublea}
		(a*x)*a = -\tfrac{1}{2}(a*a)*x+\tfrac{1}{12}\Big( 2\langle a,x \rangle a +\langle a,a \rangle x\Big).
	\end{equation}
	Summing over $B$ and plugging in \cref{doublex,doublea} into \cref{eq:lem:identity2}, we get
	\begin{multline*}
		\sum_{x\in B}\Bigg(\left( a*x\right)*\left( a*x\right) \\- \tfrac{1}{6}\Big( \langle a,x \rangle\left( a*x\right) +\langle x,a*a \rangle x +\langle a,x*x \rangle a\Big)
		+  a*\left(\left( a*x\right)*x\right) +  x*\left(\left( a*x\right)*a\right)\Bigg)
		\\ =\sum_{x\in B}\Big(\left( a*x\right)*\left( a*x\right)-\tfrac{1}{2}x*(x*(a*a))\Big) - \tfrac{1}{3}a*a + \tfrac{7}{3}a*a +\tfrac{1}{6}a*a
		\\ = \sum_{x\in B}\left( a*x\right)*\left( a*x\right) +a*a
		 =0\text{.}\qedhere 
	\end{multline*}
\end{proof}
This lemma is enough to define our next multiplication.
\begin{proposition}\label{prop:f4odot2}
	Define for $ab,cd\in \Sq^2W$ the product
	\begin{multline}\label{eq:f4odot2}
	ab\odot_2 cd   \coloneqq \tfrac{1}{2}\left(\langle a,c\rangle \left( bd*\id_W + 2bd \right) \right. +\langle a,d \rangle \left( 	bc*\id_W + 2bc  \right)
	 \\+ \langle b,c \rangle \left(  ad*\id_W + 2ad  \right)    \left. + \langle b,d \rangle \left(  ac*\id_W + 2ac  \right) \right) \\ - \tfrac{20}{3\cdot26}( \langle a,c\rangle\langle b,d\rangle + \langle a,d\rangle \langle b,c \rangle)\id_W \text{,}
\end{multline}
	Then $\odot_2$ restricted to $\sigma(V)$ is a well-defined, $F_4$-equivariant commutative product on $\sigma(V)$.
\end{proposition}
\begin{proof}
	First we prove that
	\begin{multline*}
		\left(\sum_{i}a_ib_i\right)\odot_2|_{\sigma(V)} \left(\sum_{j}c_jd_j\right) \\\coloneqq \sum_{i,j}\tfrac{1}{2}\Pi\Big(\langle a_i,c_j \rangle \left( b_id_j*\id_W + 2b_id_j  \right) + \langle a_i,d_j \rangle \left( 	b_ic_j*\id_W + 2b_ic_j  \right)\\ + \langle b_i,c_j \rangle \left(  a_id_j*\id_W + 2a_id_j  \right) + \langle b_i,d_j \rangle \left(  a_ic_j*\id_W + 2a_ic_j  \right) \Big)\text{,}
	\end{multline*}
	where $\Pi\colon \sigma(V)\oplus k \to \sigma(V)$ is the projection to $\sigma(V)$, is an $F_4$-equivariant product with image in $\sigma(V)$.
	We can assume $k$ is algebraically closed, so we can find an orthonormal basis for $W$.
	For any orthonormal basis $X$ of $W$ we have by \cref{lem:constructionssymmsquare}
	\[ \sum_{x\in X} xx = \id_{W} \text{.} \]
	So for any $b,d\in W$, we get
	\[bd*\id_W = 2\sum_{x\in X} (b*x)(d*x)\text{.}  \] 
	By this expression, \cref{prop:identitymult2} then tells us that $\odot_2$ is well-defined on $\sigma(V)$. 
	
	As ${\id_W}^{g} = \id_W$ for any $g\in F_4$, $\odot_2$ is also $F_4$-equivariant. It is clear from the definition that this product is also commutative.
	Next we prove that the definition above coincides with \cref{eq:f4odot2}.
	We can write
\begin{multline*}
	ab\odot_2 cd =  \tfrac{1}{2}\Big(\langle a,c \rangle \left( bd*\id_W + 2bd  \right) +\langle a,d \rangle \left( 	bc*\id_W + 2bc  \right)\\
	 + \langle b,c \rangle \left(  ad*\id_W + 2ad  \right) + \langle b,d \rangle \left(  ac*\id_W + 2ac  \right) \Big) - f_2(ab,cd)\id_W\text{,}
\end{multline*}
where $f_2$ is a symmetric bilinear $F_4$-equivariant form on $V$. But we only have one such form (up to a scalar multiple), so the equality
\[ f_2(ab,cd)= \lambda_2( \langle a,c\rangle\langle b,d\rangle + \langle a,d\rangle \langle b,c \rangle) \]
holds for a certain $\lambda_2 \in k$. By computing this product for $a_1a_1$ and $b_1b_1$ where $a$ and $b$ are arbitrary isotropic octonions, we obtain
$\lambda_2 = \tfrac{20}{3\cdot26}$.
\end{proof}

\begin{proposition}\label{prop:f4products}
		The space of commutative $F_4$-equivariant products on the representation of highest weight $2\omega_4$ is spanned by $\odot_1$ and $\odot_2$, defined in \cref{prop:f4odot1,prop:f4odot2}.
\end{proposition}
\begin{proof}
	By \cref{reptheoryremark}, the commutative product space is $2$-dimensional. One can check easily that $\odot_1$ and $\odot_2$ are linearly independent.
\end{proof}

Now all that remains is to determine the parameters belonging to the multiplication.

\subsection{Calculating parameters}\label{parametersection}
We will now compare the result of \cref{prop:excompf4} to the multiplications found in the previous section.
\begin{lemma}\label{ex:f4}
For $k$ an algebraically closed field and two arbitrary isotropic octonions $a,b\in \Oct$ and $j\in \{1,2,3\}$ we have
\begin{multline*}
	a_ja_j\odot_1 b_jb_j = \tfrac{\langle a ,b\rangle^2}{9}(\mathbf{1}_i+\mathbf{1}_k-2\mathbf{1}_j)(\mathbf{1}_i+\mathbf{1}_k-2\mathbf{1}_j)+\tfrac{2\langle a,b \rangle}{3}a_jb_j - \tfrac{4\langle a ,b\rangle^2}{3\cdot 26}\id_{W}\text{,}
\end{multline*}
and
\begin{multline*}
	a_ja_j\odot_2 b_jb_j \\
	= \tfrac{14\langle a,b\rangle}{3}a_jb_j + \tfrac{ \langle a,b\rangle^2}{9}(\mathbf{1}_i+\mathbf{1}_k-2\mathbf{1}_j)(\mathbf{1}_i+\mathbf{1}_k-2\mathbf{1}_j) +\tfrac{\langle a,b\rangle^2}{2}(I_i+I_k) - \tfrac{40\langle a,b\rangle^2}{3\cdot 26}\id_W \text{.}
\end{multline*}

\end{lemma}
\begin{proof}
	The first formula is a straightforward computation from \cref{prop:f4odot1}.
	
	For the second multiplication $\odot_2$, we get from \cref{prop:f4odot2} (with $X$ an orthonormal basis for $W$) 
\begin{align*}
	&a_ja_j\odot_2 b_jb_j =4\langle a,b \rangle \left( \sum_{x\in X} (a_j*x)(b_j*x) + a_jb_j \right)  - \tfrac{40\langle a,b\rangle^2}{3\cdot 26}\id_W 
\end{align*}
Let $B$ be an orthonormal basis for the octonions. We can use the following basis to compute this product:
\[ X= \left\{b_i \,\middle|\, b\in B, i \in \{1,2,3\} \right\}\cup \{ \tfrac{1}{\sqrt{2}}(\mathbf{1}_1 - \mathbf{1}_2), \tfrac{1}{\sqrt{6}}(\mathbf{1}_1 + \mathbf{1}_2 - 2 \mathbf{1}_3) \}. \]
Then, using \cref{f4technicallemma,albertmultiplication} we get the formula in the statement of the lemma.
\end{proof}

Comparing these two multiplications to our obtained result for $\star$ in \cref{prop:excompf4}, we can see the following.

\begin{theorem}\label{algF4}
	 Let $\star\colon \Sq^2W\times \Sq^2W \to \Sq^2W$ be the multiplication given by
	\begin{align*}
		ab\star cd 
		&= \tfrac{1}{24}(\langle a,c \rangle (bd*\id_W +bd) + \langle a,d \rangle (bc*\id_W +bc)+ \langle b,c \rangle (ad*\id_W+ad) +\langle b,d \rangle (ac*\id_W +ac) ) \\
		&\hspace{3ex} + \tfrac{1}{36}( \langle a,c \rangle bd+\langle a,d \rangle bc+\langle b,c \rangle ad+\langle b,d \rangle ac +\langle a,b\rangle cd + \langle c,d\rangle ab\text )\\
		&\hspace{3ex} - \tfrac{1}{6}ab*cd  - \tfrac{1}{72}(\langle a,c\rangle\langle b,d\rangle + \langle a,d\rangle\langle b,c\rangle)\id_W\text{,}
	\end{align*}
	and let $\displaystyle A =  \left\{ \sum_i X_iY_i \in \Sq^2W \,\middle\vert\, \sum_i X_i*Y_i = 0  \right\}$.
	Then $a\star a' \in A$ for all $a,a'\in A$, and the  algebra $A(\mathfrak{f}_4)$ is isomorphic to $A$ with multiplication $\star$.
\end{theorem}
\begin{proof}
	As in \cref{identificationag2}, we may assume $k$ is algebraically closed.
	As in the $G_2$ case, multiplication is of the form \begin{multline}\label{eq:multmaintheoremf4}
		ab\star cd = \left(\tfrac{1}{26^2}\langle a,b\rangle \langle c,d \rangle +f\left(ab-\tfrac{1}{26}\langle a,b\rangle\id_W,cd-\tfrac{1}{26}\langle c,d \rangle\id_W\right)\right)\id_W\\ +\tfrac{1}{26}\langle a,b\rangle \left(cd-\tfrac{1}{26}\langle c,d \rangle\id_W\right)+\tfrac{1}{26}\langle c,d \rangle \left(ab-\tfrac{1}{26}\langle a,b\rangle\id_W\right) \\+ \left(ab-\tfrac{1}{26}\langle a,b\rangle\id_W\right)\odot \left(cd-\tfrac{1}{26}\langle c,d \rangle\id_W\right)\text{.}
	\end{multline}
	By \cref{ex:f4,prop:f4odot1,prop:f4odot2}, the multiplication $\odot$ is given by
	\[ \odot = \tfrac{1}{12}\odot_2 - \tfrac{1}{12}\odot_1 \text{,} \]
	or explicitly
	\begin{align*}
		ab\odot cd &= \tfrac{1}{24}(\langle a,c \rangle bd*\id_W  + \langle a,d \rangle bc*\id_W + \langle b,c \rangle ad*\id_W+\langle b,d \rangle ac*\id_W  ) \\
		&\hspace{3ex} + \tfrac{5}{72}( \langle a,c \rangle bd+\langle a,d \rangle bc+\langle b,c \rangle ad+\langle b,d \rangle ac )\\
		&\hspace{3ex} - \tfrac{1}{6}ab*cd \\
		&\hspace{3ex} - \tfrac{1}{2\cdot 26}(\langle a,c\rangle\langle b,d\rangle + \langle a,d\rangle\langle b,c\rangle)\id_W\text{.}
	\end{align*}
	Computing $ab\odot \id_{W}$ gives
	\begin{align*}
		ab\odot \id_W &= \sum_{x\in X}\tfrac{1}{24}( \langle a,x \rangle bx*\id_V  + \langle a,x \rangle bx*\id_V + \langle b,x \rangle ax*\id_V+\langle b,x \rangle ax*\id_V  ) \\
		&\hspace{3ex} + \sum_{x\in X} \tfrac{5}{72}( \langle a,x \rangle bx+\langle a,x \rangle bx+\langle b,x \rangle ax+\langle b,x \rangle ax )\\
		&\hspace{3ex} - \tfrac{1}{6}ab*\id_W \\
		&\hspace{3ex} - \sum_{x\in X}\tfrac{1}{2\cdot 26}(\langle a,x\rangle\langle b,x\rangle + \langle a,x\rangle\langle b,x\rangle)\id_W \\
		&= \tfrac{5}{18}ab - \tfrac{1}{26}\langle a,b \rangle\id_{W}\text{.}
	\end{align*}
	We get
	\begin{multline}\label{eq:maintheoremproductf4}
		\left( ab-\tfrac{1}{26}\langle a,b\rangle \id_{W}  \right)\odot \left( cd-\tfrac{1}{26}\langle c,d\rangle \id_{W} \right) \\= \tfrac{1}{24}(\langle a,c \rangle bd*\id_W  + \langle a,d \rangle bc*\id_W + \langle b,c \rangle ad*\id_W+\langle b,d \rangle ac*\id_W  ) \\
		+ \tfrac{5}{72}( \langle a,c \rangle bd+\langle a,d \rangle bc+\langle b,c \rangle ad+\langle b,d \rangle ac )
		 \\
		  - \tfrac{1}{6}ab*cd- \tfrac{5}{18\cdot26}\langle a,b\rangle cd - \tfrac{5}{18\cdot26} \langle c,d\rangle ab \\
		- \tfrac{1}{2\cdot 26}(\langle a,c\rangle\langle b,d\rangle + \langle a,d\rangle\langle b,c\rangle)\id_W + \tfrac{23}{18\cdot 26^2}\langle a,b\rangle \langle c,d\rangle \id_W\text{.}
	\end{multline}
	The bilinear form $f(ab,cd)$ is given by
	\[ f(ab,cd) = \left(\tfrac{1}{26}-\tfrac{1}{36}\right)\tfrac{1}{2}( \langle a,c\rangle\langle b,d\rangle +\langle a,d\rangle\langle b,c \rangle )\text{.} \]	
	Then for arbitrary $a,b,c,d\in W$, we get
	\begin{multline}\label{eq:maintheorembilformf4}
		f\left( ab-\tfrac{1}{26}\langle a,b\rangle \id_{W},  cd-\tfrac{1}{26}\langle c,d\rangle \id_{W} \right)
		\\= \left(\tfrac{1}{26}-\tfrac{1}{36}\right)\left( \tfrac{1}{2}( \langle a,c\rangle\langle b,d\rangle +\langle a,d\rangle\langle b,c \rangle ) - \tfrac{1}{26}\langle a,b\rangle\langle c,d\rangle \right)\text{.}
	\end{multline}
	Plugging \cref{eq:maintheorembilformf4,eq:maintheoremproductf4} into \cref{eq:multmaintheoremf4} and cancelling out terms, we obtain the formula in the statement of the theorem.
\end{proof}
\begin{remark}
	As an anonymous referee correctly points out, we can write $ab*\id_W = 2\sum_{x\in B} (a*x)(b*x) = 2R^*_a\bullet R^*_b$, since $*$ is symmetric with respect to $\langle \cdot,\cdot\rangle$. But by \cref{lem:springerveldkamprojected},  this is precisely equal to $-R^*_{a*b}+ \frac{1}{3}ab + \frac{1}{6}\langle a,b\rangle \id_W$. This provides another way of describing the expression for the product in \cref{algF4}.
\end{remark}
\section*{Conflicts of interest}
The author declares that they have no conflict of interest.
\bibliographystyle{alpha}
\bibliography{sources}

\newcommand{\etalchar}[1]{$^{#1}$}
\begin{thebibliography}{DMVC21}

\bibitem[BGL14]{1diminvariants}
Hernando Bermudez, Skip Garibaldi, and Victor Larsen.
\newblock Linear preservers and representations with a 1-dimensional ring of
  invariants.
\newblock {\em Transactions of the American Mathematical Society},
  366(9):4755--4780, 2014.

\bibitem[Bou02]{Bourbaki46}
Nicolas Bourbaki.
\newblock {\em Lie groups and {L}ie algebras. {C}hapters 4--6}.
\newblock Elements of Mathematics (Berlin). Springer-Verlag, Berlin, 2002.
\newblock Translated from the 1968 French original by Andrew Pressley.

\bibitem[Bou05]{Bourbaki7-9}
Nicolas Bourbaki.
\newblock {\em Lie groups and {L}ie algebras. {C}hapters 7--9}.
\newblock Elements of Mathematics (Berlin). Springer-Verlag, Berlin, 2005.
\newblock Translated from the 1975 and 1982 French originals by Andrew
  Pressley.

\bibitem[CG21]{chayet2020class}
Maurice Chayet and Skip Garibaldi.
\newblock A class of continuous non-associative algebras arising from algebraic
  groups including e8.
\newblock {\em Forum Math. Sigma}, 9:Paper No. e6, 2021.

\bibitem[DMVC21]{DMVC21}
Tom De~Medts and Michiel Van~Couwenberghe.
\newblock Non-associative frobenius algebras for simply laced chevalley groups.
\newblock {\em Trans. Amer. Math. Soc.}, 374:8715--8774, 2021.

\bibitem[DSJ{\etalchar{+}}20]{sagemath}
The~Sage Developers, William Stein, David Joyner, David Kohel, John Cremona,
  and Burçin Eröcal.
\newblock Sagemath, version 9.0, 2020.

\bibitem[FH91]{FultonHarris}
William Fulton and Joe Harris.
\newblock {\em Representation theory}, volume 129 of {\em Graduate Texts in
  Mathematics}.
\newblock Springer-Verlag, New York, 1991.
\newblock A first course, Readings in Mathematics.

\bibitem[GG15]{GG15}
Skip Garibaldi and Robert~M. Guralnick.
\newblock Simple groups stabilizing polynomials.
\newblock {\em Forum of Mathematics, Pi}, 3:e3, 2015.

\bibitem[Hum75]{HumphreysLAG}
James~E. Humphreys.
\newblock {\em Linear algebraic groups}.
\newblock Springer-Verlag, New York-Heidelberg, 1975.
\newblock Graduate Texts in Mathematics, No. 21.

\bibitem[Jan03]{jantzenrepalggroups}
Jens~Carsten Jantzen.
\newblock {\em Representations of algebraic groups}, volume 107 of {\em
  Mathematical Surveys and Monographs}.
\newblock American Mathematical Society, Providence, RI, second edition, 2003.

\bibitem[Lü01]{lub01}
Frank Lübeck.
\newblock Small degree representations of finite chevalley groups in defining
  characteristic.
\newblock {\em LMS Journal of Computation and Mathematics}, 4:135–169, 2001.

\bibitem[Mal44]{Malcevconjugacy}
A.~Malcev.
\newblock On semi-simple subgroups of {L}ie groups.
\newblock {\em Bull. Acad. Sci. URSS. S\'{e}r. Math. [Izvestia Akad. Nauk
  SSSR]}, 8:143--174, 1944.

\bibitem[Mil17]{milne2017algebraic}
James~S Milne.
\newblock {\em Algebraic groups: the theory of group schemes of finite type
  over a field}, volume 170.
\newblock Cambridge University Press, 2017.

\bibitem[Myu13]{myung2013malcev}
Hyo~Chul Myung.
\newblock {\em Malcev-admissible algebras}, volume~64.
\newblock Springer Science \& Business Media, 2013.

\bibitem[Sch95]{nonassocalg}
Richard~D. Schafer.
\newblock {\em An introduction to nonassociative algebras}.
\newblock Dover Publications, Inc., New York, 1995.
\newblock Corrected reprint of the 1966 original.

\bibitem[Sei87]{Seitz}
Gary~M. Seitz.
\newblock The maximal subgroups of classical algebraic groups.
\newblock {\em Mem. Amer. Math. Soc.}, 67(365):iv+286, 1987.

\bibitem[SV00]{SpringerVeldkamp}
Tonny~A. Springer and Ferdinand~D. Veldkamp.
\newblock {\em Octonions, {J}ordan algebras and exceptional groups}.
\newblock Springer Monographs in Mathematics. Springer-Verlag, Berlin, 2000.

\end{thebibliography}

\end{document}